\documentclass[a4paper,10pt,oneside]{amsart}
\usepackage{algorithm}
\usepackage[utf8]{inputenc}
\usepackage{mathtools}
\usepackage{pgfplots}

\newtheorem{theoremA}{Theorem}

\newtheorem{corollaryA}{Corollary}

\newtheorem*{theorem*}{Theorem}

\newtheorem{lemma}{Lemma}
\newtheorem*{lemma*}{Lemma}

\newtheorem{corollary}[lemma]{Corollary}
\newtheorem{proposition}[lemma]{Proposition}

\theoremstyle{definition}
\newtheorem{definition}{Definition}
\newtheorem{example}{Example}
\newtheorem*{assumption*}{Assmuption}
\newtheorem{remark}[lemma]{Remark}

\renewcommand{\i}{{\iota}}
\DeclareMathOperator{\ord}{ord}

\DeclareMathOperator{\supp}{supp}

\DeclareMathOperator{\ini}{In}
\DeclareMathOperator{\res}{Res}
\DeclareMathOperator{\tres}{\overline{\res}}
\DeclareMathOperator{\bres}{\underline{\res}}

\makeatletter
\@namedef{subjclassname@2020}{%
 \textup{2020} Mathematics Subject Classification}
\makeatother

\title[Solutions of differential and $q$-difference equations]{Complexity of Puiseux solutions of differential 
and $q$-difference equations of order and degree one}

\author{J. Cano Torres}
\email{jcano@agt.uva.es}
\address{Universidad de Valladolid, Spain.}
\author{P. Fortuny Ayuso}
\email{fortunypedro@uniovi.es}
\address{Universidad de Oviedo, Spain.}
\author{J. Rib\'{o}n}
\email{jribon@id.uff.br}
\address{Universidade Federal Fluminense, Brazil.}
\date{\today}

\usepackage{enumerate}

\definecolor{nuevo}{RGB}{0,0,120}

\newboolean{javier}
\setboolean{javier}{true}
\newcommand{\javier}{\ifthenelse{\boolean{javier}}{\color{red}\setboolean{javier}{false}}{\color{black}\setboolean{javier}{true}}} 
\DeclareMathOperator{\Top}{Top}
\DeclareMathOperator{\Bot}{Bot}

\newboolean{pedro}
\setboolean{pedro}{true}
\newcommand{\pedro}{\ifthenelse{\boolean{pedro}}{\color{blue}\setboolean{pedro}{false}}{\color{black}\setboolean{pedro}{true}}}

\newboolean{jose}
\setboolean{jose}{true}
\newcommand{\jose}{\ifthenelse{\boolean{jose}}{\color{violet}\setboolean{jose}{false}}{\color{black}\setboolean{jose}{true}}}

\thanks{\textbf{Financial Support}. All authors partially supported by the Ministerio de Ciencia e Innovación (Spain),
Project Id. PID2019-105621GB-I00. The third author was also supported by the Brazilian National Council for Scientific and 
Technological Development - CNPq, Proc.  308838/2019-0, and FAPERJ - 
Fundação Carlos Chagas Filho de Amparo à Pesquisa do Estado do Rio de Janeiro, Processo
SEI 260003/003548/2022.}

\begin{document}
	
\begin{abstract}
	We relate the complexity of both differential and $q$-difference equations of order one and degree one and their solutions.
	Our point of view is to show that if the solutions are complicated, the initial equation is complicated too. 
	In this spirit, we bound from below an invariant of the differential or $q$-difference equation, the height of its Newton polygon, 
	in terms of the characteristic factors of a solution. The differential and the $q$-difference cases are treated in a unified way.
\end{abstract}

\maketitle

\section{Introduction} 
The ``Poincar\'{e} problem'', which consists in finding an upper bound for the algebraic degree of an invariant curve of a polynomial differential equation in the complex plane \cite{Poincare2}, has greatly influenced the study of singular holomorphic foliations. See \cite{campillo1997proximity,Cano-Fortuny-Ribon-2020,Carnicer,cavalier_lehmann_2006,Cerveau-Neto-1991,COUTINHO2006603,el2002plane,10.1145/1005285.1005309,galindo2014poincare,genzmer2018local,lei2005algebraic,pereira2002poincare,10.2307/2661388,soares2001geometry,WALCHER200051} for just a possibly biased collection of relevant citations. A related problem, which might be called the \emph{local Poincar\'{e} problem}, consists in trying to find upper bounds for the \emph{multiplicity} at a point of an invariant analytic curve of a holomorphic foliation defined in a germ of complex surface. As a matter of fact, the solution to this problem in the non-dicritical case is an essential part of the proof of the main result in \cite{Carnicer}: 
$\deg (\Gamma) \leq \deg ({\mathcal F}) + 2$, where $\Gamma$ is an invariant algebraic curve of a holomorphic foliation 
${\mathcal F}$ in $\mathbb{CP}^2$ with no dicritical singularities and $\deg$ stands for the degree. In this work, we focus on this local problem, and solve it \emph{at the same time} for differential and $q$-difference equations, as we shall show.

As the original Poincar\'{e} problem is stated for differential
equations, the usual techniques for solving it are geometric in nature
and derived from the general theory of singularities of curves and of
plane holomorphic foliations. There is, however, a less known powerful
tool called the \emph{Newton polygon} or \emph{diagram}
\cite{Newton-OO-4}, introduced by the renowned physicist and
mathematician as a tool for computing solutions of algebraic
equations, and later applied by Cramer \cite{Cramer} for computing power series
$y=\sum_{\i>0} a_{\i}x^{\i}$ with rational exponents that are
solutions of analytic equations $f(x,y)=0$. This tool, which is purely
algorithmic and makes no reference to the geometric nature of the
problem, can be applied to any two-variable problem involving power
series for which one seeks a solution in terms of well-ordered power
series in one of the variables. In \cite{CanoJ2}, this technique is
applied to differential equations in two variables, whereas in
\cite{barbe2020qalgebraic}, it is used in the context of $q$-algebraic
equations. The modifications required for these applications are
minimal, and one can unify the arguments and state general results
regardless of the context.

 Specifically, let $\sigma$ denote either the differential
operator $y(x)\mapsto dy(x)/dx$ or the $q$-difference one
$y(x)\mapsto y(qx)$, and let
$P\equiv A(x,y) + B(x,y)\sigma(y)$ be a ``first order and first
degree'' polynomial in the operator $\sigma$: $A(x,y)$ and $B(x,y)$ are power series over $\mathbb{C}$, and a solution of $P=0$ is a power series with rational exponents $s(x)\in \cup_{m \in {\mathbb N}} {\mathbb C}[[x^{\frac{1}{m}}]]$ such that $P(x,s(x)) \equiv 0$. Given such an analytic differential or $q$-difference equation in two variables, our main objective is computing an upper bound for the 
complexity of a power series solution $s(x)$ in terms of some property of the original equation. Obviously it is critical to consider the right notion of complexity. 
Notice that if $\gamma$ is the local, possibly formal, plane branch given by the parametrization $(x,s(x))$, then there is no way to bound the multiplicity of $\gamma$ in terms of algebraic invariants of the equation and specifically, its multiplicity: 
the differential equation $nx dy - my dx =0$,
with $m, n \in {\mathbb N}$ and $\gcd (m,n)=1$, has multiplicity $1$ but its power series solutions are $s(x)=cx^{\frac{m}{n}}$ for any $c$. The curve $(x, cx^{m/n})$ has multiplicity $\min (m,n)$ if $c \in {\mathbb C}^{*}$, that can be arbitrarily large. For $q$-difference equations, the same issue occurs with the equations $y-q^{\frac{n}{m}}\sigma(y)=0$ of multiplicity $0$, where $\sigma$ is the $q$-difference operator: the solutions are $s(x)=cx^{\frac{m}{n}}$ for $c \in \mathbb{C}$. 
This problem cannot be overcome, which leads to seeking a different criterion for the complexity of a Puiseux power series solution.

In this paper, we consider the characteristic exponents of $s(x)$ as the measure of such
complexity, following the point of view of \cite{Cano-Fortuny-Ribon-2020}. Notice that
the characteristic exponents of $s(x)$ are intimately related to the Puiseux
characteristic of the curve $\Gamma$ defined by $(x,s(x))$ \cite{Wall} but, when $\Gamma$
is tangent to $x=0$, one has to use the well-known \emph{inversion formula}
\cite{Zariski-equi-iii,Abhyankar-inversion} in order to compute ones from the others. The
characteristic exponents are significant invariants for germs of plane curves: for
instance, their number, which has come to be called {\it genus}, is deeply related to the
topology of the curve $\gamma$, as it measures the levels of interlacing of the associated
knot \cite{Wall}. Given
$s(x) \in \cup_{m \in {\mathbb N}} {\mathbb C}[[x^{\frac{1}{m}}]]$ with $s(0)=0$, let
$g \geq 0$ be the genus of $s(x)$ and $n$ its multiplicity. One can derive, from the
characteristic exponents, positive integers $r_1,\ldots, r_g$, which we later call the \emph{characteristic
 factors} in Definition \ref{def:puiseux-exponents-and-factors}, that are
greater than $1$ and such that if $n$ is the least common denominator of the exponents of
$s(x)$, then $r_1\cdots r_g=n$. Our results hinge on
these factors and the notion of \emph{dicritical exponent}. Assume $s(x)=\sum a_ix^{i/m}$
is a solution of $P=0$. Roughly speaking, an exponent $k/m$ of $s(x)$ 
is \emph{dicritical} if for all but finitely many $c\in \mathbb{C}$, there exits another
solution $s_{c}(x)$ of $P=0$ such that $s_c(x)-\sum_{i <k} a_ix^{i/m} =x^{k/m} u(x)$ where $u(0) = c$. We include a brief excursus in subsection \ref{subs:dicritical} relating our definition with
the classical definition of dicritical divisor of a singular holomorphic
foliation.

In what follows,
 $P\equiv A(x,y) + B(x,y)\sigma(y)$ is an operator with
$A(x,y), B(x,y) \in {\mathbb C}[[x,y]]$ such that $A(0,0)=B(0,0)=0$, and
$s(x) \in \cup_{m \in {\mathbb N}} {\mathbb C}[[x^{\frac{1}{m}}]]$
with $s(0)=0$, is a solution of $P=0$ with $r_1,\ldots, r_g$ its
characteristic factors. After constructing the Newton diagram $\mathcal{N}(P)$, we shall attach to $P$ and $s(x)$ several invariants: $H(P)$, the \emph{height of $P$}, which is the topmost vertex of $\mathcal{N}(P)$; the \emph{multiplicity of $P$}, $\nu_0(P)$, which is the minimum multiplicity of $A(x,y)$ and $B(x,y)$; and a number $H(P,s(x))$, a kind of \emph{relative height}, which is, roughly speaking, the topmost vertex of the part of $\mathcal{N}(P)$ corresponding to the order of $s(x)$, $\ord(s(x))$. By definition, we have
 \begin{equation}
 \label{eq:Intro:H(P)geqH(P,s(x)}
 H(P)\geq H(P,s(x)),\quad\text{for any }s(x),
 \end{equation}
and also, 
 \begin{equation}
 \label{eq:Intro:multiplicity_and_H(P)}
 \nu_{0}(P)+1\geq H(P,s(x)),\quad\text{if }\ord (s(x))\geq 1.
 \end{equation}
 Our main results
 provide bounds of $H(P)$ and of
$\nu_0(P)$ from below in terms of the characteristic factors $r_1,\ldots, r_g$.
\begin{theoremA}\label{the:main} 
If $1\leq i_1<\ldots< i_d \leq g$ is the sequence of indices of
dicritical characteristic exponents of $s(x)$, then
\begin{equation*} 
 H(P)\geq H(P,s(x)) \geq
 \prod_{j=1}^gr_j -
 \sum_{k=1}^{d}\left( \prod_{j=1}^{i_{k}}r_j - \prod_{j=1}^{i_{k}-1}r_j\right).
\end{equation*}
If, moreover, $\ord (s(x)) \geq 1$, then $\nu_0(P)+1$
is greater than or equal to the right hand side of the inequality.
\end{theoremA} %
By convention and unless expressly stated otherwise, we consider that
an empty sum is equal to $0$ and that an empty product is equal to
$1$. In particular, the right hand side is equal to $1$ if $g=0$.
Note that if no characteristic exponent corresponds to a dicritical
element, then Theorem \ref{the:main} reads
\begin{displaymath}
 H(P)\geq H(P,s(x)) \geq r_1\cdots r_g.
\end{displaymath}
The inequality can be improved with every instance of
consecutive dicritical characteristic exponents
(Lemma \ref{lem:no-consecutive-dicritical}). In this way, one obtains a simplified form in which no assumptions need to be made about the dicritical exponents. 

\begin{corollaryA} \label{cor:a}
 Let $r_1,\ldots, r_g$ be the characteristic factors of $s(x)$. Then 
\begin{equation} \label{H:differential} 
 H(P)\geq H(P,s(x)) > \prod_{j=1}^{g-1} r_j - \prod_{j=1}^{g-2}r_j .
\end{equation}
If, moreover, $\ord (s(x)) \geq 1$, then $\nu_0(P)$
is greater than or equal to the right hand side of the inequality.
 
\end{corollaryA}
We can
improve this result in the differential, (see
\cite{Cano-Fortuny-Ribon-2020}), in the \emph{generic} $q$-difference
and in the contracting, i.e. $|q|<1$, $q$-difference cases. To this end,
we introduce the concept of \emph{reasonable equations} (see Definition
\ref{def:reasonable}), that {} encompasses {} the previous cases. 
\begin{theoremA} \label{teo:reasonable}
 If $s(x)$ is a Puiseux solution of genus $g$ of the reasonable equation $P=0$, then
\begin{equation*} 
 H(P)\geq H(P,s(x)) \geq r_1\cdots r_{g-1} .
 \end{equation*}
If, moreover, $\ord (s(x)) \geq 1$, then one also has $\nu_0(P)+1\geq r_1\cdots r_{g-1}$.
 
\end{theoremA}
 As a consequence we get also a bound for the genus $g$ of a solution
$s(x)$, namely $g\leq 1+\log_2(H(P))$, and if $\ord (s(x)) \geq 1$, then $g\leq 1+\log(\nu_0(P)+1)$. 

We end our paper showing how the bound for the multiplicity of a differential equation found in \cite{Cano-Fortuny-Ribon-2020} can be obtained exclusively by means of the Newton polygon using our technique. Let $\nu_{0}(\mathcal{F})$ be the multiplicity at $0 \in \mathbb{C}^{2}$ of the singular foliation defined by the differential equation $A(x,y)dx+B(x,y)dy =0$, assuming $A(x,y)$ and $B(x,y)$ have no common factors. 
\begin{corollaryA}\label{cor:multiplicity-inequality} 
Let ${\mathcal F}$ be a germ of singular holomorphic foliation in a neighborhood of the origin in ${\mathbb C}^{2}$ that has a formal 
irreducible invariant curve $\Gamma$
whose characteristic factors are $r_1,\ldots, r_g$. Then, we obtain 
\begin{equation*}
 \nu_{0}(\mathcal{F}) \geq r_1 \cdots r_{g-1},
 \end{equation*}
 where an empty product is $1$.
\end{corollaryA} 

To summarize, we apply the Newton polygon technique \emph{simultaneously} to both
differential and $q$-difference equations in order to obtain lower bounds for the height of the Newton polygon in terms of the 
characteristic factors of a solution $s(x)$ that parametrises an irreducible curve.
Similar results were first proved in
\cite{Cano-Fortuny-Ribon-2020} in the differential case using geometric techniques related
to the desingularization of the curve defined by $s(x)$. Those bounds are valid in the 
differential and the generic $q$-difference case, which includes the contracting
($|q|<1$) case. In the case of a non-generic non-contracting $q$-difference equation $P$,
those lower bounds for $H(P)$ are just somewhat worse. A final section is devoted to improving the bound in the case of differential equations, and obtaining the same bound as in \cite{Cano-Fortuny-Ribon-2020}, just with the Newton polygon technique.

\textbf{Acknowledgment.} This work was greatly improved thanks to the many suggestions of an anonymous referee.

\section{Notation and preliminary results}
From now on, a complex number $q \in {\mathbb C}^{*}$ is chosen with $|q|\neq 1$, and also a specific determination of the complex logarithm, which we shall denote $\log(z)$ for $z\in \mathbb{C}$ whenever required. There is no indetermination, as the reader will notice.

Let $\sigma$ be one of the following operators on the set of Puiseux series over $x$ with non-negative exponents:
\begin{equation}
 \label{eq:sigma}
 \sigma\bigg(\sum_{i\geq 0} a_ix^{i/n}\bigg) = \left\{
 \begin{array}{l}
 \displaystyle\sum_{i\geq 0} \frac{i}{n} a_i x^{(i-n)/n} \\[15pt]
 \displaystyle\sum_{i\geq 0} q^{i/n}a_i x^{i/n}
 \end{array}
 \right.
\end{equation}
The first one will be called the \emph{differential operator}, the second one the \emph{$q$-difference operator}. The operator $\sigma$ is extended to a variable $y$ giving $\sigma(y)=y_1$ (the variable ``operated''). This way, we can write any differential equation of order and degree one, or any $q$-difference equation in which the $q$-difference operation only appears to degree and order one as
\begin{equation}
 \label{eq:initial-equation}
 A(x,y) + B(x,y)y_1 = 0.
\end{equation}

Before defining the concept of solution, we gather all the equations
we are going to study under a single concept:
\begin{definition}\label{def:covered-equation-and-solution}
 An $m$-\emph{covered equation} is an Equation (\ref{eq:initial-equation}) where:
 \begin{enumerate}
 \item Both $A(x,y)$ and $B(x,y)$ are formal power series in ${\mathbb C}[[x^{\frac{1}{m}},y]]$ with $A(0,0)=B(0,0)=0$;
 \item $y_1$ stands for $\sigma(y)$, where $\sigma$ is any of the operators in Equation (\ref{eq:sigma}). 
 \end{enumerate}
 We say that the equation is {\it covered} if it is $m$-covered for some $m \in {\mathbb N}$. 
 A \emph{solution} of such an equation is a Puiseux series $s(x)$ in $\cup_{m \in {\mathbb N}} {\mathbb C}[[x^{\frac{1}{m}}]]$ 
 such that (as Puiseux series)
 \begin{equation*}
 A(x,s(x)) + B(x,s(x))\sigma(s(x)) = 0
 \end{equation*}
 holds, 
 where $\sigma$ is the appropriate operator. Finally, the \emph{order} $o_{\sigma}$ of $\sigma$ is $0$ for the $q$-difference operator and $1$ for the differential operator.
\end{definition}

Our aim is to use the Newton-Puiseux polygon (from now on just \emph{Newton polygon}) to relate the complexity of the solutions of a $1$-covered Equation 
(\ref{eq:initial-equation}) to some specific invariant. 
Along the way, we carry out some auxiliary operations that transform Equation (\ref{eq:initial-equation}) into $m$-covered equations for $m$ possibly higher than $1$: this explains why Definition \ref{def:covered-equation-and-solution} is relevant. 
From now on, we fix a Puiseux power series
\begin{equation}\label{eq:s(x)}
 s(x) = \sum_{i\geq 1}a_i x^{i/n} \in
 \bigcup_{m \in {\mathbb N}} {\mathbb C}[[x^{\frac{1}{m}}]],
\end{equation}
where $n$ is the minimal $m \in {\mathbb N}$ such that
$s(x) \in {\mathbb C}[[x^{\frac{1}{m}}]]$. Indeed, if
$s \not \equiv 0$, $n$ is the least common denominator of the
exponents having non-zero coefficient: in technical terms, $s(x)$ is a
\emph{reduced power series} and the series $s(x)$ is a formal power
series if and only if $n=1$. In the case of differential equations,
we shall consider, in the last section, the analytic branch $\Gamma$
associated with $s(x)$, and relate its multiplicity to the notion of
multiplicity of the associated foliation. This requires us to perform
a change of coordinates, that will be introduced in the last section,
which has no equivalent in the $q$-difference case.

There are two cases, either $n=1$ or there exists a first index $e_1$ such that $a_{e_1}\neq 0$ and $e_1/n\not\in \mathbb{Z}$.
In the former case we define $g=0$ whereas in the latter case we write 
\begin{equation*}
 \frac{e_1}{n} = \frac{p_1}{r_1}
\end{equation*}
with $p_1$, $r_1$ mutually prime with $r_1 \geq 2$. Assuming $e_i$, $p_i$, $r_i$ are defined, either $r_1 \hdots r_{i} = n$ and we define $g=i$ or
there exists a first index $e_{i+1}$ such that $a_{e_{i+1}}\neq 0$ and $e_{i+1}/n\not\in \frac{1}{r_1\cdots r_i}\mathbb{Z}$, and write
\begin{equation} \label{def:char}
 r_1\cdots r_{i} \frac{e_{i+1}}{n} = \frac{p_{i+1}}{r_{i+1}}.
\end{equation}
 with $\gcd (p_{i+1}, r_{i+1})=1$ and $r_{i+1} \in {\mathbb N}_{\geq 2}$.
This construction ends at some $g\geq 0$ when $r_1\cdots r_g=n$.

As we shall work with power series and only in the case of foliations we shall consider
the associated germ of analytic curve, we use the following definition, associated with
$s(x)$ and not with the germ of curve $\Gamma$ defined by it. However, the \emph{genus} is
a measure both of the complexity of $s(x)$ and of the topological complexity of $\Gamma$ \cite{Wall}:
\begin{definition}\label{def:puiseux-exponents-and-factors}
 The numbers $e_{1}, \hdots, e_{g}$ are called the \emph{characteristic exponents} of $s(x)$, and the factors $r_1,\ldots, r_g$ will be called the \emph{characteristic factors}. 
 The number $g$ of characteristic exponents is the so called \emph{genus} of $s(x)$. If $g=0$ then $n=1$ and the Puiseux series $s(x)$ is said \emph{nonsingular}.
\end{definition}

\subsection{The Newton polygon}Given a covered equation as (\ref{eq:initial-equation}), the \emph{Newton polygon or diagram} 
is a graphical help for computing its solutions. Its construction follows.

Fix a covered equation $P= P(x,y,y_1)\equiv A(x,y)+B(x,y)y_1=0$ and write
\begin{equation}\label{eq:coefficients-sum}
 A(x,y) = \sum a_{\i j}x^\i y^{j} = \sum A_{\i j}x^\i y^{j}, \,\,\,
 B(x,y) = \sum b_{\i j}x^\i y^j = \sum B_{\i - o_{\sigma} \, j+1}x^\i y^{j} 
\end{equation}
where $\i\in \mathbb{Q}_{\geq 0}$ and $j\in \mathbb{N}$, where we use $\i$ instead of $i$ to emphasize that it may not be an integer. 
The \emph{supports} of $A(x,y)$ and $B(x,y)$ are the sets 
\begin{equation*}
 \supp(A) = \left\{
 (\i, j) : A_{\i j}\neq 0 
 \right\} \ \mathrm{and} \ 
 \supp(B) = \left\{
 (\i, j) : B_{\i j}\neq 0 
 \right\}
\end{equation*}
respectively. 
 The support of $B$ is obtained from $\{ (\i, j) : b_{\i j}\neq 0 \}$ by pushing one step up, 
 because of the factor $y_1$, and one step left
 in the case of differential equations, because $\sigma$ decreases the order of each monomial $x^{\mu}$ by one.

\begin{definition}\label{def:cloud-of-points}
 The \emph{cloud of points} of $P$ is the set $\mathcal{C}(P) = \supp (A) \cup \supp(B)$.
 \end{definition}

Consider the following subset of $\mathbb{R}_{\geq -1} \times \mathbb{R}_{\geq 0}$:
\begin{equation*}
 \mathcal{Q}(P) = \bigcup_{(\i, j)\in \mathcal{C}(P)} (\i, j) +
 \big(\mathbb{R}_{\geq 0} \times \mathbb{R}_{\geq 0}\big),
\end{equation*}
where we place a positive quadrant at each point of the cloud.
\begin{definition}\label{def:newton-polygon}
 The \emph{Newton polygon} $\mathcal{N}(P)$ of $P$ is the convex envelope of $\mathcal{Q}(P)$.
\end{definition}

\begin{example}
 Consider the equation
 \begin{equation}\label{eq:example}
 P \equiv y^{4} + x^3y^{3}+ xy^2 - x^3y + x^5 + (xy^3 - x^2y)y_1.
 \end{equation}
 Its Newton polygon is shown in Figure \ref{fig:newton-polygon-1}. The points $(1,4)$ and $(2,2)$, the unfilled circles, correspond to $ (xy^3 - x^2y)y_1$ in the
 $q$-difference case. The Newton polygon, however, is the same, in this case, for both equations.
\begin{figure}[h!]
 \centering
 \includegraphics{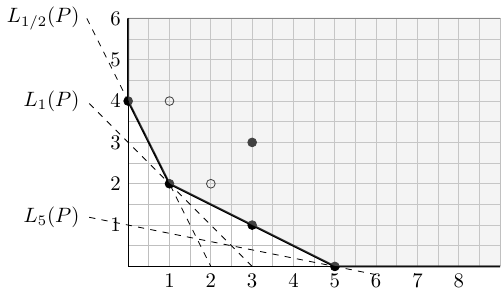}
 \caption{Cloud of points, Newton polygon and some supporting lines (see Def. \ref{def:supporting}) of the equation $P=0$ in \eqref{eq:example}. The two unfilled points correspond only to the $q$-difference case, whereas the filled ones correspond to both cases.}
 \label{fig:newton-polygon-1}
\end{figure}

\end{example}

\subsection{Newton polygon and solutions}
The interest of this construction will become (hopefully) apparent at the end of this section. Take $s(x)$ as in \eqref{eq:s(x)}. From now on, we fix a covered equation
\begin{equation}
 \label{eq:starting-equation}
 P \equiv A(x,y) + B(x,y)y_1 = 0,
\end{equation}
and denote by $\sigma$ the corresponding operator.

\begin{definition}\label{def:truncation}
 The \emph{$k$-th truncation} of $s(x)$ is the Puiseux series
 \begin{equation*}
 s_k(x) = \sum_{ 0 < i \leq k }a_ix^{i/n}.
 \end{equation*}
\end{definition}
Note that the truncation \emph{includes} the term $a_{k}x^{k/n}$ of $s(x)$. By convention, $s_0(x)=0$.
 

\begin{definition}\label{def:substitution-of-equation}
 The \emph{$k$-th substitution} (of $s(x)$, but this will always be implicit) in $P$ is the equation
 \begin{equation*}
 P_k\equiv A(x, y + s_k(x)) + B(x, y + s_k(x))(y_1+\sigma(s_k(x))) = A^k(x,y) + B^k(x,y)y_{1},
 \end{equation*}
 where $s_k(x)$ is the $k$-th truncation of $s(x)$ (thus, $P_0=P$). The \emph{total substitution} of $s$ in $P$ is the equation
 \begin{equation*}
 P_{\infty} \equiv A(x, y + s(x)) + B(x, y + s(x))(y_1+\sigma(s(x))).
 \end{equation*}
As $P_k=A^{k}(x,y)+B^{k}(x,y)y_1$, the expressions $A^k_{\iota j}$, $B^k_{\iota j}$ will denote the corresponding coefficients of $A^k(x,y)$ and $B^k(x,y)$, following (\ref{eq:coefficients-sum}).
\end{definition}

Notice that if $R(x,y,y_1)=P_k$, then
\begin{equation*}
 R\big(x,y+a_{k+1}x^{(k+1)/n},y_1+\sigma(a_{k+1}x^{(k+1)/n})\big)=P_{k+1},
\end{equation*}
 and the definition of $P_k$ can be made iterative, substitution by substitution. This is Newton and Cramer's construction, which allowed 
 the latter to find approximate solutions of algebraic equations. There are some geometric concepts required for the proper application of the Newton polygon to solving covered equations. 

\begin{definition}\label{def:supporting}
 Given $\mu\in \mathbb{R}_{>0}$, let $L_{\mu}(P)$ denote the line 
 \begin{equation*}
 L_{\mu}(P) \equiv \bigg\{ (i,j) \in {\mathbb R}^{2} : j = \frac{-1}{\mu} i + \alpha \bigg\} 
 \end{equation*}
 with $\alpha$ maximum satisfying the following property: if
 $L^+_{\mu}(P) \equiv \big\{(i,j)\ \big|\ j \geq -i/\mu + \alpha \big\} $, then
 $ \mathcal{N}(P)\subset L^{+}_{\mu}(P) $. This line will be called the \emph{supporting
 line of $\mathcal{N}(P)$ of co-slope ~$\mu$}.
\end{definition}
Notice that $L_{\mu}(P)\cap \mathcal{N}(P)$ is either a vertex of $\mathcal{N}(P)$ or a
side:
\begin{definition}
 The \emph{element of co-slope $\mu$} of $\mathcal{N}(P)$ is
 \begin{equation*}
 E_{P,\mu} = L_{\mu}(P) \cap \mathcal{N}(P),
 \end{equation*}
 and it will be called either the \emph{vertex of co-slope $\mu$} or the \emph{side of
 co-slope $\mu$} if $E_{P,\mu}$ is a single point or otherwise. We shall denote
 $E_{k , \mu} = E_{P_{k},\mu}$ and, because $E_{k,k/n}$ will be our main concern, $E_k = E_{k, k/n}$.
\end{definition}
 We stress the fact that $E_k$ is the element of co-slope $k/n$ \emph{after applying the $k$-th substitution}, whereas $E_{k-1,k/n}$ is the element of the same co-slope \emph{just before} that substitution has been applied. We refer the reader to the later example of Subsection \ref{subs:example} for this important distinction. 
\begin{example}\label{ex:elements}
 In Figure \ref{fig:newton-polygon-1}, the elements $E_{P,\mu}$ for $\mu\in[1/2, 2]$ are the following. To begin with, $E_{P,1/2}$ is the segment joining $(0,4)$ and $(1,2)$ that corresponds to the dashed line $L_{1/2}(P)$). Then, for $1/2<\mu<2$, $E_{P,\mu}$ is just the vertex $(1,2)$. Finally, $E_{P,2}$ is the segment from $(1,2)$ to $(5,0)$, which contains the point $(3,1)$.
\end{example}
 Let $\mu$ be a co-slope and $E_{P,\mu}$ the corresponding element of $P$. We can unify the notation for differential and $q$-difference equations using the $\delta$ coefficient:
\begin{definition}\label{def:delta}
 The $\delta$ coefficient corresponding to the co-slope $\mu$ is the number:
 \begin{equation*}
 \delta_{\mu} = \left\{
 \begin{array}{ll}
 \mu & \text{if}\ P\ \text{is a differential equation}\\
 q^{\mu} & \text{if}\ P\ \text{is a}\ q\text{-difference equation}
 \end{array}
 \right.
 \end{equation*}
\end{definition}
This\ allows us to unify the first key concept:
\begin{definition}
 The \emph{initial polynomial} of $P$ of co-slope $\mu$ is $\Phi_{P,\mu}(C)$, given by:
 \begin{equation*}
 \Phi_{P,\mu}(C) = \sum_{(\iota, j)\in E_{P,\mu}}
 (A_{\iota j} + \delta_{\mu} B_{\iota j})C^{j} . 
 \end{equation*}
 When working with $P_k$, we shall normally use the notation $\Phi_{k , \mu}(C)$ instead of $\Phi_{P_{k},\mu}(C)$.
 \end{definition}
 \begin{definition}
If the initial polynomial is identically zero, i.e.
$\Phi_{ k-1 ,k/n}(C)\equiv 0$, then the exponent $k/n$ of $s(x)$, the element $E_{ k-1 ,k/n}$ and
the co-slope $k/n$ are called \emph{dicritical}.
\end{definition} 
 In subsection \ref{subs:dicritical} we relate this notion to the geometric concept of \emph{dicritical foliation}. 

The following results are classical for differential equations \cite{CanoJ2,CanoJ} and trivially extended to $q$-difference equations (see \cite{barbe2020qalgebraic,cano2012power} for instance). Fix $P$ and a solution $s(x)$ as above.

\begin{lemma}\label{lem:coefficients-are-zeros}
 For any $k>0$, we have
 \begin{equation*}
 \Phi_{ k-1 ,k/n}(a_k)=0.
\end{equation*}
\end{lemma} 

A cornerstone of the method is that the operation with the term $a_{k}$ does not modify
the Newton polygon ``up to the part corresponding to $a_{k-1}$'' :

 \begin{lemma}\label{lem:same-newton-polygon}
 If $(\iota, t)$ is the topmost point of $E_{k-1,k/n}$, then, for any
 $ l \geq k$:
 \begin{equation*}
 \mathcal{N}(P_{k-1}) \cap (\mathbb{R}_{\geq -1}\times \mathbb{R}_{\geq t}) =
 \mathcal{N}(P_{ l}) \cap (\mathbb{R}_{\geq -1}\times \mathbb{R}_{\geq t}),
 \end{equation*}
 that is, both polygons are equal at $(\iota ,t)$ and above. Even more, if
 \begin{equation*}
 P_{k-1} = A^{k-1}(x,y) + B^{k-1}(x,y)y_{1},
 \end{equation*}
 then $A^{k-1}_{\iota t}=A^{ l}_{\iota t}$ and $B^{k-1}_{\iota t}=B^{ l}_{\iota t}$
 for any $ l \geq k-1$.
 As a consequence, for all $k\geq 1$ and $l\geq k$, we have
 \begin{equation}\label{eq:E-k-stays}
 E_{k}=E_{l,k/n} = E_{P_{\infty},k/n} .
 \end{equation}
\end{lemma} 

Finally, solutions are characterized by their ``flattening'' of the Newton polygon from below:
\begin{theorem*}[see \cite{CanoJ2, barbe2020qalgebraic}]\label{the:solutions-flat}
 Let $P\equiv A(x,y)+B(x,y)y_1=0$ be a covered equation and $s(x)$ a Puiseux series with $a_0=0$ (no independent term). The following statements are equivalent:
 \begin{enumerate}
 \item The power series $s(x)$ is a solution of $P$,
 \item The power series $0$ is a solution of $P_{\infty}$,
 \item The Newton polygon of $P_{\infty}$ has a horizontal side at height greater than $0$.
 \end{enumerate}
\end{theorem*}

Therefore, if $s(x)=\sum a_{i}x^{i/n}$ is a solution of $P$, then $a_i$ is a root of the
corresponding initial polynomial $\Phi_{ i-1 ,i/n}(C)$. If this holds for each $i$, then
$s(x)$ is indeed a solution of $P$. Due to Lemma \ref{lem:same-newton-polygon}, the
polygon construction is, thus, an iterative process in which each coefficient $a_{i}$ is a
zero of the initial polynomial of the unique element of $P_{i-1}$ of co-slope
$i/n$. Furthermore, also by Lemma \ref{lem:same-newton-polygon}, this latter element is \emph{to the right and not above} the element
of $P_{i-1}$ of co-slope $(i-1)/n$.

The following definition covers all the main invariants associated to $P$ and $s(x)$.
 \begin{definition}\label{def:top-and-bottom}
 Let $P\equiv A(x,y)+B(x,y)y_1=0$ be a covered equation and $s(x)$ be a
 Puiseux series with order $\ord (s(x))>0$.
 The \emph{height of $P$}, denoted $H(P)$, is the ordinate of the leftmost vertex of $\mathcal{N}(P)$.
 Consider a co-slope $\mu$ and the corresponding element $E_{P,\mu}$
 of $P$ of co-slope $\mu$. The \emph{top (or height)} of $E_{P,\mu}$,
 is the highest ordinate of the points of $E_{P,\mu}$, and the
 \emph{bottom} of $E_{P,\mu}$ is the lowest. They will be denoted as
 $\Top(E_{P,\mu})$ and $\Bot(E_{P,\mu})$, respectively. We denote $H(P,s(x))=\Top(E_{P,\mu})$ for
 $\mu=\ord(s(x))$. Finally, the \emph{multiplicity}
 of $P$ at the origin is
 \begin{displaymath}
 \nu_0(P)=\min\{\ord_{(x,y)} (A(x,y)),\ord_{(x,y)} (B(x,y))\}.
 \end{displaymath}
\end{definition}
\begin{remark}\label{rem:heights} 
As $\mathcal{N}(P)$ has a finite number of sides, we have 
\begin{itemize}
\item The map $\mu \to \Top(E_{P,\mu})$ is a decreasing function from ${\mathbb R}^{+}$ to $\mathbb{Z}$.
\item For any $\mu\in \mathbb{R}^{+}$, we have $H(P) \geq
 \Top(E_{P,\mu})$, in particular for any $s(x)$,
 \[H(P)\geq \Top(E_{P,\mu})\geq H(P,s(x)),\quad 0<\mu\leq
 \ord(s(x)).\]
If, moreover, $\ord (s(x)) \geq 1$, then
\begin{displaymath}
 \nu_0(P)+1 \geq \Top(E_{P,1})\geq H(P,s(x)).
\end{displaymath}
\item There is $\mu_{0}>0$ such that $H(P)=\Top(E_{P,\mu})$ for any $0< \mu<\mu_0$.
\end{itemize}
\end{remark} 
We are interested, for a $1$-covered equation $P$, in bounding $H(P)$ and $H(P,s(x))$
from below in terms of the characteristic factors of a solution.
In lay terms, we wish to prove that an equation with a complicated solution must already
be ``complicated'' where the complexity is measured by $H(P)$ or $H(P,s(x))$. In the last section, devoted to the case of differential
equations, we shall see how, up to a linear change of coordinates, the multiplicity of the associated foliation at $0\in \mathbb{C}^2$
is greater than or equal to $H(P,s(x)) -1$. 
We shall use our general results to bound this multiplicity from below.

Lemma \ref{lem:same-newton-polygon} implies the following property which we shall use freely:
\begin{lemma}\label{lem:top-less-bottom}
 For any $k\geq 0$,
 \begin{equation*}
 \begin{split}
 \Top(E_{k-1,k/n})= &\Top(E_{k})\geq \Bot(E_{k}) \geq \\
 &\Top(E_{k,(k+1)/n}) =
 \Top(E_{k+1}).
\end{split}
 \end{equation*}
\end{lemma}

 \begin{remark}\label{rem:top-less-than-bottom}
 Let $(\iota_{0}, 0)$ be the point of intersection of the $x$-axis and $L_{k/n} (P_{k})$.
 Later on we shall see that $A_{\iota_{0} 0}^{k} = \Phi_{ k-1 ,k/n}(a_k)$ in Equation \eqref{eq:alpha-beta-descent-1}.
 Thus, Lemma \ref{lem:coefficients-are-zeros} gives $\Top (E_k) \geq \Bot(E_{k}) \geq 1$ for any
 $k\geq 1$: after each substitution, the element $E_{k}$ does not meet the $x$-axis.
 On the other hand, Lemma \ref{lem:top-less-bottom}
 implies $\Bot(E_{k})\geq \Top(E_{l-1, l/n})$ for $k< l$. Moreover, if for
 some $j$ with $k<j< l$, the element $E_{j}$ contains more than one vertex, then
 $\Bot(E_{k}) > \Top(E_{ l-1, l/n})$ because, in this case,
 \[
 \Bot(E_{k})\geq \Top(E_{j})>\Bot(E_{j})\geq \Top(E_{ l-1, l/n}).
 \]
\end{remark}
This last remark is quite relevant because it will provide a \emph{descent} argument for our bounds.
We shall see in Lemma \ref{lem:new-puiseux-exponent-new-side} that all characteristic exponents $k/n$, except possibly the last one, 
give rise to sides in the Newton polygon, i.e. $\Top (E_k) > \Bot (E_k)$.
Moreover, one of our main results (Proposition \ref{pro:descent-step}) provides a qualitative estimate of the gap between 
$\Bot(E_k)$ and $\Top (E_k)$.
 

 \subsection{An example}\label{subs:example}
For the benefit of the reader, we include an exhaustive example in this section, in order to clarify the technique, the notation and some of the results.

Consider the differential equation associated with the following polynomial:
\begin{equation}
 \begin{split}
 P_0 = y^4+4 y^3 x+5 y^2 x^2+2 y x^3+y x^4+4 x^5+x^7 + \\
 (-y^3 x-4 y^2 x^2-5 y x^3-2 x^4+3 x^5)y_1.
\end{split}
\label{eq:full-example}
\end{equation}
We know in advance -- this is the initial assumption in this work -- that $P_0$ admits a solution with the following Puiseux expansion:
\begin{equation*}
 s(x) = - x - \sqrt{11}x^{3/2} - \frac{121}{30}x^2 + \cdots
\end{equation*}
where the exponents of the remaining terms belong to $\frac{1}{2}\mathbb{Z}$ and are greater than $2$. Thus, $e_1 = 3$ is the single characteristic exponent of $s(x)$.
So, setting $n=2$, we have $a_2=-1$, $a_3=-\sqrt{11}$, $a_4=-\frac{121}{30}$. The clouds of points, Newton polygons corresponding to each substitution, and their respective 
elements are depicted in Figure \ref{fig:full-example}. These substitutions are computed in the following paragraphs. Recall that $\delta_{k/n}=k/n=k/2$ because $P_0$ is differential 
and $n=2$.

\begin{enumerate}
\item The first exponent is $1=2/2$, so that $k=2$. Thus, the relevant element of
 $\mathcal{N}(P_1)=\mathcal{N}(P_0)$ is $E_{1,2/2}$. The initial polynomial is
 \begin{equation*}
 \Phi_{1,2/2}(C) = (1-1)C^{4} + (4-4)C^3 + (5-5)C^{2} + (2-2)C \equiv 0,
 \end{equation*}
 that is, $E_{1,2/2}$ is a dicritical element.
\item Once the substitution $y= y-x$ is performed, we obtain
 \begin{equation*}
 P_2 \equiv
 y^4+x y^3+x^4 y+x^7 + (-x y^3-x^2 y^2+3 x^5)y_1,
 \end{equation*}
 whose element $E_2:=E_{2,2/2}$ of co-slope $1$ is, in this case, a shorter subsegment of
 $E_{1,2/2}$. Notice that it might have been a single point or, if $E_{1,2/2}$ were shorter, it might
 have been longer but by Remark \ref{rem:top-less-than-bottom}, $\Bot(E_{2})\geq
 1$, in any case. Following $s(x)$, the next relevant element is $E_{2,3/2}$, which corresponds to
 the single characteristic exponent $e_{1}=3$, which gives $r_1=2$. The initial
 polynomial $P_{2,3/2}(C)$ is
 \begin{equation*}
 \Phi_{2,3/2}(C) = -\frac{1}{2}C^{3} + \frac{11}{2}C,
 \end{equation*}
 whose roots are $C=0$ and $C=\pm \sqrt{11}$, that include,
 certainly, $a_3=-\sqrt{11}$.
\item Performing the substitution $y = y - \sqrt{11} x^{3/2}$ corresponding to $a_{3}x^{3/2}$, we obtain
 \begin{align*}
 &P_3 = \\
 &y^4
 -\tfrac{5 \sqrt{11}}{2}x^{{{3}\over{2}}} y^3
 -\tfrac{3 \sqrt{11}}{2}x^{{{5}\over{2}}} y^2
 +x y^3+\tfrac{33}{2} x^3 y^2
 +x^4 y
 +\tfrac{11^{{{3}\over{2}}}}{2} x^{{{9}\over{2}}} y
 -\tfrac{121}{2}x^{6}
 +x^7
 +\\
 &(-x y^3
 -x^2 y^2
 +3 \sqrt{11} x^{{{5}\over{2}}} y^2
 +2\sqrt{11} x^{{{7}\over{2}}} y
 -33 x^4 y
 -8 x^5
 +11^{{{3}\over{2}}}x^{{{11}\over{2}}}
 )y_{1},
 \end{align*}
 whose element $E_{3}$ is, in this case, of the same length as $E_{2,3/2}$ but
 contains a new point with non-integral $x$-coordinates: $(5/2,2)$. As $e_{1} =3$ is a
 characteristic exponent, there must be at least one such point in the cloud, as we shall
 show in Lemma \ref{lem:new-puiseux-exponent-new-grid}. The element $E_{3,4/2}$
 corresponding to $a_4$ is the side joining $(4,1)$ and $(6,0)$. The initial polynomial
 is
 \begin{equation*}
 \Phi_{3,4/2}(C) = -15C - \frac{121}{2},
 \end{equation*}
 whose unique root is, certainly, $a_4=-\frac{121}{30}$.
\item Finally, after the substitution corresponding to $a_4$, we obtain $P_4$, whose Newton polygon is also depicted.
\end{enumerate}
Notice that the Newton polygon $\mathcal{N}(P_{k-1})$ coincides with $\mathcal{N}(P_l)$ for $l\geq k$ from $\Top(E_{k-1,k/n})$
up, as per Lemma \ref{lem:same-newton-polygon}. Also, $E_{k,k/n}=E_{l,k/n}$ for $l\geq k$, as in the diagram corresponding to $\mathcal{N}(P_4)$. 
\noindent
\vspace*{3pt}\\
\begin{figure}
 \centering
 \begin{tabular}{cc}
 \centering
 \def\points{( 0 , 4 ) , ( 1 , 3 ) , ( 2 , 2 ) , ( 3 , 1 ) , ( 4 , 1 ) , ( 5 , 0 ) , ( 7 , 0 )}
 \renewcommand{\arraystretch}{0.1}
 \begin{tabular}{c}
 \includegraphics{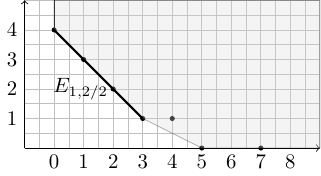}\\
 $\mathcal{N}(P_1)$.
\end{tabular}
&
 \def\points{( 0 , 4 ) , ( 1 , 3 ) , ( 4 , 1 ) , ( 7 , 0 )}
\renewcommand{\arraystretch}{0.1}
\begin{tabular}{c}
 \includegraphics{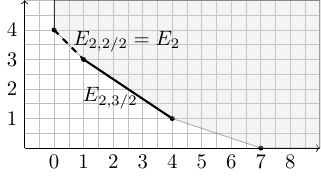}\\
 {$\mathcal{N}(P_2)$}
\end{tabular}
 \\
\\
\vspace*{10pt}
 \def\points{( 0 , 4 ) , ( 1 , 3 ) , ( 3/2 , 3 ) , ( 5/2 , 2 ) , ( 3 , 2 ) , ( 4 , 1 ) , ( 9/2 , 1 ) , ( 6 , 0 ) , ( 7 , 0 )}
\renewcommand{\arraystretch}{0.1}
\begin{tabular}{c}
 \includegraphics{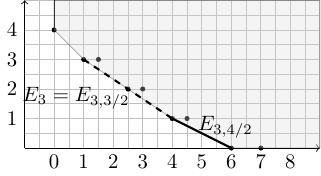}\\	
 {$\mathcal{N}(P_3)$}
\end{tabular}
&
\def\points{( 0 , 4 ) , ( 1 , 3 ) , ( 3/2 , 3 ) , ( 2 , 3 ) , ( 5/2 , 2 ) , ( 3 , 2 ) , ( 7/2 , 2 ) , ( 4 , 1 ) , ( 4 , 2 ) , ( 9/2 , 1 ) , ( 5 , 1 ) , ( 11/2 , 1 ) , ( 6 , 1 ) , ( 13/2 , 0 ) , ( 7 , 0 ) , ( 15/2 , 0 ) , ( 8 , 0 )}
\renewcommand{\arraystretch}{0.1}
\begin{tabular}{c}
 \includegraphics{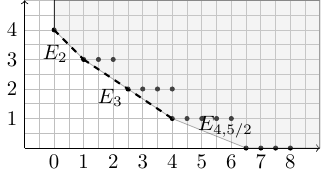}\\
 {$\mathcal{N}(P_4)$}
\end{tabular}
\end{tabular}
\caption{Newton polygons and relevant elements for $P_0$ in \eqref{eq:full-example}. Notice that $E_{k,k/n}$ is later referred to as $E_k$.}
\label{fig:full-example}
\end{figure}
\subsection{Characteristic exponents}
A remarkable property of the characteristic exponents of $s(x)$ in terms of the Newton
polygon is that each one -- except possibly the last one --, say $k/n$, gives rise to a \emph{whole side}
in $\mathcal{N}(P_k)$ and, by Lemma
\ref{lem:same-newton-polygon}, in $\mathcal{N}(P_l)$ for $l>k$. This fact, already noted in item (3) of
the example in Subsection \ref{subs:example}, and which we now prove, is essential to find
our bounds.
\begin{lemma}\label{lem:new-puiseux-exponent-new-side}
 Assume that $P$ is $1$-covered, i.e. the initial equation has only integer exponents,
 and let $s(x)$ be a solution as above. If $k=e_\ell$ for some $\ell=1,\ldots,g-1$, then
 the Newton polygon of $P_{k}$ (and, by Lemma \ref{lem:same-newton-polygon}, of $P_{j}$
 for $j>k$) has a side of co-slope $k/n$; that is, $E_{k}$ is indeed a side, not just a
 vertex. If $e_g$ is not dicritical, then the result holds also for $k=e_g$.
\end{lemma}
\begin{proof}
 Assume $k=e_{\ell}$ with $\ell \leq g$. Let $\Phi(C)=\Phi_{ k-1
 ,k/n}(C)$ be the corresponding initial polynomial. By recurrence,
 $\mathcal{C} (P_{k-1})$ is included in $\frac{1}{r_1\cdots
 r_{\ell-1}}\mathbb{Z}\times \mathbb{Z}$, so that all the points
 have abscissa with denominator at most $r_1\cdots r_{\ell-1}$. Let
 $(\iota,t)$ be the topmost vertex of $E_{ k-1 ,k/n}$, corresponding
 to the terms $A_{\iota t}^{k-1}x^{\iota}y^{t}+B_{\iota t}^{k-1}x^{
 \iota + o_{\sigma} }y^{t-1}y_1$. In particular, $t=\Top(E_{
 k-1 ,k/n})$. Performing the substitution $y=y+a_kx^{k/n}$ at the
 terms of $P_{k-1}$ corresponding to the point $(\iota,t)$,
 we get
 \begin{equation*}
 \begin{split}
 &A_{\iota t}^{k-1} x^{\iota }(y+a_kx^{k/n})^{t} +
 B_{\iota t}^{k-1} x^{\iota + o_{\sigma} }
 (y+a_kx^{k/n})^{t-1}(y_1+\sigma(a_{k}x^{k/n}))\\
 &=A_{\iota t}^{k-1} x^{\iota} y^t + B_{\iota t}^{k-1} x^{\iota + o_{\sigma} } y^{t-1}y_1 + 
 (tA_{\iota t}^{k-1} + \delta_{{k/n}} B_{\iota t}^{k-1}) a_{ k } x^{\iota +k/n}y^{t-1}\\
 &+ (t-1)B_{\iota t}^{k-1} a_kx^{\iota +k/n + o_{\sigma} } y^{t-2}y_1 + \cdots,
 \end{split}
\end{equation*}
where the dots indicate terms whose ordinates in $\mathcal{N}(P_{k})$ are strictly less than
$t-1$. 
Recall that $\delta_{k/n}$ is either $k/n$ or $q^{k/n}$, depending on $P$ being a
differential or $q$-difference equation. 
 Note that $(\iota+k/n,t-1)$ does not belong to $\mathcal{C}(P_{k-1})$ since 
$\iota+k/n\not \in \frac{1}{r_1\cdots r_{\ell-1}}\mathbb{Z}$.
Moreover, $(\iota, t)$ is the only point in 
$\mathcal{C}(P_{k-1})$ contributing to $(\iota +k/n, t-1)$ when performing the substitution.
Thus, it suffices to show that either
$(t-1)B_{\iota t}^{k-1} a_k \neq 0$ or $tA_{\iota t}^{k-1}+ \delta_{k/n} B_{\iota t}^{k-1} \neq 0$
to obtain $(\iota+k/n,t-1) \in \mathcal{C}(P_k)$.

Notice that $a_{k} \neq 0$ because $e_{\ell}$ is a characteristic exponent, so that it appears explicitly in $s(x)$. If $t>1$, then either
$B_{\iota t}^{k-1} \neq 0$, so that $(t-1)B_{\iota t}^{k-1} a_k$ is not zero, or
$(tA_{\iota t}^{k-1}+ \delta_{k/n} B_{\iota t}^{k-1})=tA_{\iota t}^{k-1}$ is not zero. {} Thus, we can consider
$t=1$ and $A_{\iota 1}^{k-1}+ \delta_{k/n} B_{\iota 1}^{k-1}=0$ from now on. In particular we obtain
$A_{\iota 1}^{k-1} \neq 0 \neq B_{\iota 1}^{k-1}$. {}

We claim that $t=1$ and $A_{\iota 1}^{k-1}+ \delta_{k/n} B_{\iota 1}^{k-1}=0$ imply that
$e_\ell$ is dicritical and $\ell=g$, finishing the proof. As $t=1$, the
condition $A_{\iota 1}^{k-1}+ \delta_{k/n} B_{\iota 1}^{k-1}=0$ implies that 
$\Phi(C) \equiv (A_{\iota 1}^{k-1}+ \delta_{k/n}B_{\iota 1}^{k-1})C + \Phi(0) \equiv \Phi (0)$. 
Since $\Phi(a_k)=0$, we deduce that $\Phi(C)\equiv 0$.
Thus, $k=e_{\ell}$ is dicritical. Also, since for $k^{\prime}\geq k$ we have
\[ 1 = \Top (E_{ k-1 ,k/n}) = \Top (E_k) \geq \Top (E_{k^{\prime}}) > 0 \]
by Lemma \ref{lem:top-less-bottom} and Remark \ref{rem:top-less-than-bottom}, we deduce 
$\Top(E_{k^{\prime}}) = 1$ for $k^{\prime}\geq k$. We claim that
$k'/n \in \frac{1}{r_1\cdots r_{\ell}}\mathbb{Z}$ for any $k' \geq k$ such that $a_{k'} \neq 0$.
This property implies $\ell = g$ by the definition of characteristic exponents.

Let us show the claim.
By construction, the property is satisfied by $k$.
Assume it holds for any $k \leq k' < k_{0}'$ and suppose, aiming at contradiction, that 
$k_{0}'/n \not \in \frac{1}{r_1\cdots r_{\ell}}\mathbb{Z}$ and $a_{k_{0}' } \neq 0$. 
Since $\delta_{k/n} \neq \delta_{k_{0}' /n}$, we get 
\[ A_{\iota 1}^{k_{0}'-1}+ \delta_{k_{0}' /n}B_{\iota 1}^{k_{0}'-1}
 =A_{\iota 1}^{k-1}+ \delta_{k_{0}' /n}B_{\iota 1}^{k-1}
 \neq A_{\iota 1}^{k-1}+ \delta_{k /n}B_{\iota 1}^{k-1}=0 \]
by Lemma \ref{lem:same-newton-polygon}.
Note that 
$\mathcal{C} (P_{k_{0}' -1}) \subset \frac{1}{r_1\cdots r_{\ell}}\mathbb{Z}\times \mathbb{Z}$ as a consequence of the induction hypothesis. 
This property, together with $\iota \in \frac{1}{r_1\cdots r_{\ell}}\mathbb{Z}$ and $k_{0}'/n \not \in \frac{1}{r_1\cdots r_{\ell}}\mathbb{Z}$ imply
$A_{\iota +k_{0}' /n \, 0}^{k_{0}' -1} =0$. Since 
\[ \Phi_{k_{0}' -1, k_{0}' /n}(C) =
 (A_{\iota 1}^{k_{0}'-1}+ \delta_{k_{0}'/n}B_{\iota 1}^{k_{0}'-1}) C
 +
 A_{\iota +k_{0}' /n \, 0}^{k_{0}'-1},\] 
$0$ is the unique root of $\Phi_{k_{0}' -1, k_{0}' /n}$, contradicting 
$\Phi_{k_{0}' -1, k_{0}' /n} (a_{k_{0}'}) =0$. 
\end{proof}

 The proof of the previous result implies that the Newton polygon after a substitution corresponding to a characteristic exponent has points in a grid
with a different scale in the variable $x$. Namely: 
\begin{lemma}\label{lem:new-puiseux-exponent-new-grid}
 Assume that $P$ is $1$-covered. 
 Let $k\in \mathbb{Z}_{>0}$, and let $r$ be the minimum integer such that $\mathcal{C}(P_{k-1})\subset \frac{1}{r}\mathbb{Z}\times \mathbb{Z}$. If $\Top(E_{k})>1$, then $k$ is a characteristic exponent if and only if the cloud of points of $P_k$ is not included in $\frac{1}{r}\mathbb{Z}\times \mathbb{Z}$.
\end{lemma}
The fact that $\Top(E_k)>1$ corresponds to the case $t>1$ in the previous proof,
which guarantees that the point $(\iota + k/n, t-1)$ belongs to $\mathcal{C}(P_k)$, and by
definition, $\iota +k/n\not\in \frac{1}{r}\mathbb{Z}$. This hypothesis is
necessary: the solutions $y=Cx^{m/n}$ of the differential equation
$P\equiv my - nxy^{\prime}=0$,
which has the single point $(0,1)$ in its cloud, leave the
cloud of points invariant after the substitution $y=y+Cx^{m/n}$.


\subsection{Decomposing the initial polynomial}
%
 By Lemma \ref{lem:coefficients-are-zeros}, the coefficient $a_k$ is always a root
of the corresponding initial polynomial $\Phi_{k-1,k/n}(C)$, this is the basis of
Newton's technique. This means that $\Bot(E_k)\geq 1$ for all $k$, see Remark \ref{rem:top-less-than-bottom}. However, we can obtain
much more information about the element $E_{k}$ and the equation $P_k$ if we study the
transformation of $E_{k-1,k/n}$ into $E_{k}$ as a parametric family depending on a complex parameter $C$, that
is, studying the substitution $y=y+Cx^{k/n}$ instead of $y=y+a_kx^{k/n}$. This is, in the
differential case, similar to studying the whole exceptional divisor corresponding to
$x^{k/n}$ and the singularities and regular points of the strict transform of the
foliation given by $P$ at the corresponding exceptional divisor. This study is carried out
by means of the \emph{$k$-th initial form}, which gathers all that information. It also
provides an additive decomposition of $\Phi_{k-1,k/n}(C)$ into two terms: one
corresponding to the algebraic part of $P_k$, and the other to the one with $y_1$. This
decomposition, and its consequences are key in further results.We follow
\cite{CanoJ2} in the definition, leaving $s(x)$ implicit.

Fix $k \geq 1$ from now in this section, and let
\[ P_{k}(C):=P_{k-1}(x, y+Cx^{k/n},y_1+\sigma(C x^{k/n})) \] 
be the
$k$-th substitution with a complex parameter $C$ instead of $a_{k}$. Writing
\[ P_k(C)=A^{k}(C)(x,y)+B^{k}(C)(x,y)y_1 = \sum 
 A^{k}_{\iota j}(C)x^{\iota} y^j + B^{k}_{\iota j}(C)x^{\iota+o_{\sigma}}y^{j-1}y_1, \]
the property $P_k(a_k)=P_k$ implies
$A_{\iota j}^k(a_k)=A_{\iota j}^k$ and $B_{\iota j}^k(a_k)=B_{\iota j}^k$ for any $(\iota,j)$, where $A_{\iota j}^k$ and $B_{\iota j}^k$ were defined in 
Definition \ref{def:substitution-of-equation}. Notice also that $P_k(0)=P_{k-1}$, $A^k_{\iota j}(0)=A^{k-1}_{\iota j}$ and $B^k_{\iota j}(0)=B^{k-1}_{\iota j}$. 
 \begin{definition}\label{def:initial-form}
 The \emph{$k$-th initial form} of $P$ for $s(x)$ is the polynomial in $C$ given by the expression:
\begin{equation*}
 \ini_{k}(C) = \sum_{(\iota,j)\in L_{k/n}(P_{k-1})}
 A^{k}_{\iota j}(C)x^{\iota} y^j + B^{k}_{\iota j}(C)x^{\iota+o_{\sigma}}y^{j-1}y_1 .
\end{equation*}
\end{definition}
From bottom to top, $\ini_{k}(C)$ can be rewritten as
\begin{equation}
 \label{eq:initial-form}
 \ini_{k}(C) = A^k_{\nu \, 0} (C) x^{\nu} + \sum_{j=1}^t 
 A^{k}_{\nu-jk/n \, j}(C)x^{\nu-jk/n}y^{j} +
 B^{k}_{\nu-jk/n \, j}(C) x^{\nu-jk/n+o_{\sigma}}y^{j-1}y_1.
\end{equation}
where $t=\Top(E_{k-1,k/n})=\Top (E_k)$ and $\nu=\overline{\iota}+t k/n$ if $(\overline{\iota},t)$ is
the topmost vertex of $E_{k-1,k/n}$. For the sake of simplicity, as $k$ is fixed
throughout all this section, we set, for the remainder of this section,
$\iota_{j}:=\nu-jk/n$, and write:
\begin{equation*}
 \ini_{k}(C) = A^k_{\iota_{0} \, 0} (C) x^{\nu} + \sum_{j=1}^t 
 A^{k}_{\iota_{j} \, j}(C)x^{\nu-jk/n}y^{j} +
 B^{k}_{\iota_{j} \, j}(C) x^{\nu-jk/n+o_{\sigma}}y^{j-1}y_1,
\end{equation*}
because we are mostly interested in $A_{\iota_{j} j} (C)$ and $B_{\iota_{j} j} (C)$, and their pairwise relations, for $j=1,\ldots, t$.

The following two polynomials decompose $\Phi_{k-1,k/n}(C)$ into two parts: one corresponding to the terms without $y_1$ in $\ini_k(C)$, and the other to those with $y_1$. They will be key in many later computations:
\begin{equation}\label{eq:alpha-beta}
 \begin{split}
 &\alpha_{k}(C)=\sum_{j=0}^tA^{k}_{\iota_{j} \, j}(0)C^j =
 \sum_{j=0}^tA^{k-1}_{\iota_{j} \, j}C^j\\
 &\beta_{k}(C) =\sum_{j=1}^tB^{k}_{\iota_{j} \, j}(0)C^{j-1} =
 \sum_{j=1}^tB^{k-1}_{\iota_{j} \, j}C^{j-1}.
\end{split}
\end{equation}
By definition, the initial polynomial $\Phi_{ k-1 ,k/n}(C)$ satisfies
\begin{equation}
 \label{eq:phi-alpha-beta}
 \Phi_{ k-1 ,k/n}(C)= \alpha_{k}(C) + \delta_{k/n}C \beta_{k}(C).
\end{equation}
The following result is the basis of the relevance of this decomposition \cite[cf. Eq. (1)]{CanoJ2}. Recall that $t=\Top(E_{k-1,k/n})=\Top(E_k)$ in this subsection:
\begin{lemma}\label{lem:alpha-beta-descent}
 With the notation above, let $f^{(r)}(C)$ denote $\displaystyle\frac{\partial^{r} f}{\partial C^r}(C)$ for any function $f(C)$. Then:
 \begin{align}
 \label{eq:alpha-beta-descent-1}
 &A^{k}_{\iota_{0} \, 0}(C) = \Phi_{ k-1 ,k/n}(C), \\
 \label{eq:alpha-beta-descent-2}
 &A^{k}_{\iota_{j} \, j}(C) = \displaystyle\frac{1}{j!}\Phi^{(j)}_{ k-1 ,k/n}(C) -
 \delta_{k/n} \frac{1}{(j-1)!}\beta_{k}^{(j-1)}(C),
 \quad j=1,\ldots, t, \\
 \label{eq:alpha-beta-descent-3}
 &B^{k}_{\iota_{j} \, j}(C) = \displaystyle\frac{1}{(j-1)!}\beta_{k}^{(j-1)}(C),
 \quad j=1,\ldots, t , \\
 \label{equ:AdB}
 & A^{k}_{\iota_{j} \, j}(C) + \delta_{k/n} B^k_{\iota_j \, j}(C) \equiv \frac{1}{j!}\Phi^{(j)}_{ k-1 ,k/n}(C) , 
 \quad j=1,\ldots, t .
\end{align}
Specifically, $A^{k}_{\iota_{t} t}(C)=A^{k-1}_{\iota_t t}$ and
$B^{k}_{\iota_{t} t}(C)=B^{k-1}_{\iota_t t}$ are independent of $C$ (this is already known by Lemma \ref{lem:same-newton-polygon}), 
and, finally $A^{k}_{\iota_{0} 0}(C)=0$ if and only if $\Phi_{k-1,k/n}(C)=0$. 
\end{lemma}
The following proof is rather technical but its gist is to transform $P_{k-1}$ and
$P_k(C)$ into new equations whose elements $E_{k-1, k/n}$ correspond to a vertical side of
the same height. When written like this, the argument becomes a direct application of
Taylor's formula.
\begin{proof} Define the polynomial
 \begin{equation*}
 I(y,y_1)=\ini_{k}(0) =
 A^{k}_{\iota_{0} \, 0} (0) x^{\nu} + \sum_{j=1}^t
 A^{k}_{\iota_{j} \, j} (0) x^{\nu-j{k/n}}y^j +
 B^{k}_{\iota_{j} \, j}(0) x^{\nu-j{k/n} + o_{\sigma}}y^{j-1}y_1 .
 \end{equation*}
Rather than compute $ \ini_{k}(C) $ directly, we
will do so in three steps. Each step is an algebraic change of
indeterminates that does not correspond to a differential or
$q$-differential change of indeterminate, but the composition of the
three does. First, we substitute $y$ by $x^{{k/n}}y$ and $y_1$ by $\delta_{k/n} x^{{k/n} - o_{\sigma}}y_1$ and notice how $\alpha_k(y)$ and $\beta_k(y)$ can be used to rewrite the result:
\begin{multline*}
 I_1(y,y_1):=I(x^{k/n}\,y, \delta_{k/n}\, x^{k/n - o_{\sigma}}\, y_1)=\\
 A^{k}_{\iota_{0} \, 0} (0) x^{\nu} +
 \sum_{j=1}^t A^{k}_{\iota_{j} \, j} (0) x^{\nu-j\,{k/n}}x^{j\,{k/n}}y^j +
 B^{k}_{\iota_{j} \, j} (0)
 x^{\nu-j{k/n}}x^{(j-1){k/n}}y^{j-1}\delta_{k/n}\,x^{{k/n}}\, y_1 \\
 =x^{\nu}\cdot\left\{ \alpha_{k}(y)+\delta_{k/n}\,\beta_{k}(y)\,y_1\right\}.
\end{multline*} 
Then we translate by $C$ both $y$ and $y_1$ and expand $\alpha(C+y)$ and $\beta(C+y)$ using Taylor's formula:
\begin{multline*}
 I_2(y,y_1):=I_1(C+y,C+y_1)=
 x^{\nu}\bigg(
 \alpha_{k}(C+y)+\delta_{k/n}\,\beta_{k}(C+y)\,(C+y_1)\bigg) =\\
 x^{\nu}\bigg(
 \sum_{j=0}^{t} \frac{1}{j!}\alpha_{k}^{(j)}(C)\,y^{j}+
 \delta_{k/n}(C+y_1)\sum_{j=0}^{t-1}\frac{1}{j!}\beta_{k}^{(j)}(C)y^{j}
 \bigg)=\\
 x^{\nu}\bigg(
 \sum_{j=0}^{t}
 \frac{1}{j!}\left(\alpha_{k}^{(j)}(C)+\delta_{k/n} \,C\,
 \beta_{k}^{(j)}(C)\right) \,y^{j}+ \delta_{k/n}
 \sum_{j=1}^{t}\frac{1}{(j-1)!}\beta_{k}^{(j-1)}(C)y^{j-1}y_1\bigg) .
\end{multline*}
Finally, we undo the first transformation:
\begin{multline*}
 \ini_{k}(C) =I_2(x^{-{k/n}}\,y,\delta_{k/n}^{-1}\, x^{-{k/n} + o_{\sigma}}\,y_1)
 =\\
 x^{\nu}\bigg(
 \sum_{j=0}^{t}
 \frac{1}{j!}\left(\alpha_{k}^{(j)}(C)+\delta_{k/n} \,C\,
 \beta_{k}^{(j)}(C)\right) \,x^{-j\,{k/n}}y^{j}\bigg. +\\
 \bigg. \delta_{k/n}
 \sum_{j=1}^{t}\frac{1}{(j-1)!}\beta_{k}^{(j-1)}(C) x^{-(j-1){k/n}}
 y^{j-1} \delta_{k/n}^{-1} x^{-{k/n} + o_{\sigma}} y_1\bigg)=\\
 \sum_{j=0}^{t}
 \frac{1}{j!}\left(\alpha_{k}^{(j)}(C)+\delta_{k/n} \,C\,
 \beta_{k}^{(j)}(C)\right) \,x^{\nu-j\,{k/n}}y^{j} 
 +
 \sum_{j=1}^{t}\frac{1}{(j-1)!}\beta_{k}^{(j-1)}(C) x^{\nu-j{k/n} + o_{\sigma} }
 y^{j-1} y_1.
\end{multline*}
From the last equalities we can already infer:
\begin{align*} A^{k}_{\iota_{0} \, 0}(C)&=\alpha_{k}(C)+\delta_{k/n}\,C\,\beta_{k}(C)=\Phi_{k-1,k/n}(C),\\
 A^{k}_{\iota_{j} \, j}(C)&=\frac{1}{j!}\left(\alpha_{k}^{(j)}(C)+\delta_{k/n} \,C\,
 \beta_{k}^{(j)}(C)\right),\quad j=1,\ldots,t , \\
 B^{k}_{\iota_{j} \, j}(C)&=\frac{1}{(j-1)!}\beta_{k}^{(j-1)}(C),\quad j=1,\ldots,t.
\end{align*}
By induction on $j$, we obtain 
\begin{displaymath}
 \Phi_{k-1,k/n}^{(j)}(C)=\alpha_{k}^{(j)}(C)+\delta_{k/n} \,C\, \beta_{k}^{(j)}(C)+j\,\delta_{k/n}\,\beta_{k}^{(j-1)}(C),
\end{displaymath}
whence
\begin{displaymath}
A^k_{\iota_{j} \, j}(C)=\frac{1}{j!}\Phi_{k-1,k/n}^{(j)}(C)-\frac{\delta_{k/n}}{(j-1)!}\beta_{k}^{(j-1)}(C),
\end{displaymath}
as desired. Equation \eqref{equ:AdB} is an immediate consequence of Equations \eqref{eq:alpha-beta-descent-2} and \eqref{eq:alpha-beta-descent-3}. 
\end{proof} 
Recall that $P_{k}=P_k(a_{k})=P_{k-1}(x,y+a_kx^{k/n},y_1+\sigma(a_kx^{k/n}))$. In the following result, we see 
how the information about $ \ini_{k}(C) $ provided by Lemma \ref{lem:alpha-beta-descent} 
allows to better understand the properties of the particular case $C=a_k$
and more precisely of the element
$E_k$ obtained after the $k$-th substitution. 
\begin{corollary}\label{cor:height-multiplicity}
Let $b=\Bot(E_{k})$. With the notation of Lemma \ref{lem:alpha-beta-descent}, the following statements hold:
 \begin{enumerate}
 \item In any case, the multiplicity of $a_k$ as a root of $\beta_{k}(C)$ is at least
 $b-1$.
 \item If $\Phi_{ k-1 ,k/n}(C)\not\equiv 0$ (non-dicritical case), then the multiplicity
 of $a_{k}$ as a root of $\Phi_{ k-1 ,k/n}(C)$ is at least $b$.
 \item If $\Phi_{k-1,k/n}(C)\equiv 0$ (dicritical case), then
 \begin{equation}\label{eq:dicritical-residue}
 A^{k}_{\iota_{j} \, j}(C) + \delta_{k/n} B^k_{\iota_j \, j}(C) = 0
 \end{equation}
 for all $C\in \mathbb{C}$ and $j=1,\ldots, \Top(E_k)$, and in particular for $C=a_k$.
 \end{enumerate}
\end{corollary}
\begin{proof} 
 The last statement is an immediate consequence of Equation \eqref{equ:AdB} and the dicritical condition $\Phi_{k-1,k/n}(C)\equiv 0$.

There are no
 points in $E_k$ with ordinate less than $b$ by definition. Since 
 \begin{equation} \label{eq:alpha-beta-descent-3-aux}
 \beta_{k}^{(j-1)}(a_{k}) = (j-1)!B^{k}_{\iota_{j} \, j}(a_{k}) = 0
 \end{equation}
 for $j=0,\ldots, b-1$ by Equation \eqref{eq:alpha-beta-descent-3}, we deduce Statement (1). Moreover, we obtain 
 $\Phi^{(j)}_{ k-1 ,k/n}(C) =0$ for any $0 \leq j \leq b-1$ by Equation \eqref{equ:AdB}.
 Statement (2) follows.
\end{proof}

\subsection{An excursus on dicritical elements}\label{subs:dicritical} In the arguments to come, the apparition of a
dicritical characteristic exponent is problematic. In the setting of differential
equations, this mirrors the fact that dicritical divisors appearing in the reduction of
singularities of a holomorphic $1$-form ``complicate'' the combinatorial structure
of the residues and indices associated with the exceptional divisors (see \cite{Camacho-Sad-Coloquio}, for instance).

Consider the family of analytic branches $s_{k,c}'(x)$ given by the $k$-truncation of $s(x)$ with $a_k$ replaced with 
a complex parameter $c$, with $c\neq 0$:
\begin{equation*}
 s_{k,c}'(x) = \sum_{i=1}^{k-1}a_ix^{i/n} + cx^{k/n}.
\end{equation*}
This family of curves admits a common desingularization in a sequence $\pi_{k}$ of point blow-ups, ending at an exceptional divisor
which, for the sake of economy, we shall call $D_k$. Each nonzero $c\in \mathbb{C}$ corresponds to a point $Q_{c}\in D_{k}$, and through $Q_{c}$ passes a single nonsingular curve $\overline{\Gamma}_c$, transverse to $D_k$, such that the parametrization of
$\pi_k(\overline{\Gamma}_c)$ coincides with $s_{k,c}'(x)$.

Let $P=A(x,y)+B(x,y)y_1$ be differential, $\mathcal{F}$ be the foliation associated with
$\omega=A(x,y)dx+B(x,y)dy$. Let 
$\mathcal{F}_k$ be the strict transform of $\mathcal{F}$ by $\pi_k$. In the
non-dicritical case, i.e. $\Phi_{k-1,k/n}(C) \not \equiv 0$, the divisor $D_k$ is invariant by ${\mathcal F}_{k}$. 
The roots $c$ of $\Phi_{k-1,k/n}(C)$ 
correspond precisely with 
the singular points $Q_{c}$ of $\mathcal{F}_k$ in $D_k$.

The dicritical case, where $\Phi_{k-1,k/n}(C)\equiv 0$, is totally different. Indeed, it 
is easy to prove that for those $c \in \mathbb{C}^{*}$ such that $\Bot(E_{k})=1$, there exists a power series $s_{k,c}(x)$ in ${\mathbb C}[[x^{\frac{1}{n}}]]$, 
that coincides with $s_{k,c}'(x)$ up to order $\leq k/n$, which is a solution of $P$ (see \cite{cano2012power}, for instance). In such a case, 
the ${\mathcal F}$-invariant curve $\Gamma_c$, whose parametrization is $s_{k,c}(x)$, satisfies that its strict transform 
$\tilde{\Gamma}_c$ by $\pi_{k}$ intersects $D_k$ at $Q_c$, $Q_c$ is a regular point of ${\mathcal F}_k$
and $\tilde{\Gamma}_{c}$ is a smooth curve transverse to $D_k$. 
However, there may be special points in $D_{k}$, that is, values
of $c$, where $\mathcal{F}_k$ is either singular or tangent to $D_k$, and these are the ones we need to control, as $a_k$ may be one of them. It
is these points which present a challenge. 
The following straightforward consequence of
Lemma \ref{lem:alpha-beta-descent} asserts that these challenging points are finite in
number.
\begin{corollary}\label{cor:dicritical-generic-bottom-1}
 Assume $E_{k-1,k/n}$ is dicritical. Consider the point $(\iota_{1},1)$ of the line $L_{k/n} (P_k)$ of height $1$.
 Then
 \begin{equation*}
 B^{k}_{\iota_{1} 1}(C)\not\equiv 0.
 \end{equation*}
 As a consequence, if $\Bot(E_k)>1$, then $a_k$ is a root of the polynomial
 $B^{k}_{\iota_{1} 1}(C)$.
\end{corollary}
\begin{proof}
 If $B^{k}_{\iota_{1} 1}(C)\equiv 0$, then $\beta_k( C) \equiv 0$ by Equation \eqref{eq:alpha-beta-descent-3}.
 We obtain $B^{k}_{\iota_{t} t} =0$, where $(\iota_{t},t)=\Top(E_{k})$, by applying again Equation \eqref{eq:alpha-beta-descent-3}. 
 Since $(\iota_{t},t)$ is a vertex of $P_k$, we deduce $A^k_{\iota_{t} t}\neq 0$.
 As $E_{k-1,k/n}$ is dicritical, and
 $\delta_{k/n}\neq 0$, Equation (3) of Corollary \ref{cor:height-multiplicity} implies that
 both $A^k_{\iota_{t} t}$ and $B^k_{\iota_{t} t}$ are nonzero, contradicting $B^{k}_{\iota_{t} t} =0$.
\end{proof}

Finally, by their nature, dicritical exponents impose strict relations between the
coefficients in $A(x,y)$ and those in $B(x,y)$ belonging to the corresponding
element. This is precisely item (3) in Corollary \ref{cor:height-multiplicity}: whenever
$A^{k}_{\iota j},B^{k}_{\iota j}$ come from a dicritical element $E_{k-1,k/n}$, then
$A^k_{\iota j}=-\delta_{k/n}B^k_{\iota j}$. This equality will play an essential role in all
our arguments in the dicritical case. Among many other things, it will prevent the
existence of immediately consecutive dicritical elements (Lemma 13), and, even more, the
existence of consecutive elements $E_k,E_{k+1},\ldots, E_{l}$ with $E_k,E_l$ dicritical,
and the ones in between satisfying a ``bad'' property (Lemma \ref{lem:inter-dicritical}). Remarkably, this latter fact is true for differential equations and
\emph{most} $q$-difference equations, and this is the only point in which both types of
equations differ. We shall collect the differential and these generic $q$-difference
equations under the concept of \emph{reasonable} equation (Definition
\ref{def:reasonable}).

In summary, dicriticalness is, in some sense, a complication but
its very nature imposes conditions which can be taken advantage of to overcome most of
said complication.

\section{Main Result}
\setcounter{theoremA}{0}
\setcounter{corollaryA}{0}

We show Theorem \ref{the:main}, Corollary \ref{cor:a} and Theorem \ref{teo:reasonable} 
in this section. As a starting point, let $P$ be a 1-covered equation -- that is, with only integral exponents -- and let
\begin{equation*}
 s(x) = \sum_{i\geq 1}a_ix^{i/n}
\end{equation*}
be a solution of $P$ in Puiseux form. Let $e_1,\ldots, e_g$ be the characteristic
exponents of $s(x)$ and $r_1,\ldots, r_g$ be the characteristic factors.

Before proceeding, we provide some definitions and notation which will simplify the statements.

\begin{definition}\label{def:factor-of-exponent}
 For an integer $k\geq 1$, the \emph{factor corresponding to $k$}, denoted $\rho_k$, is defined as follows:
 \begin{itemize}
 \item If $k$ is not a characteristic exponent, then $\rho_k=1$,
 \item otherwise, if $k$ is equal to the characteristic exponent $e_{l}$, then $\rho_k=r_l$.
 \end{itemize}
 That is, $\rho_{e_l}=r_l$ and $\rho_k=1$ if $k$ is not a characteristic exponent.
\end{definition}
The main relation between $\Top(E_{k})$ and $\Bot(E_{k})$ -- notice that this is \emph{after} the substitution of the term $a_k x^{k/n}$ -- 
when $k$ is a characteristic exponent is:
\begin{proposition}\label{pro:descent-step}
 Let $k \in {\mathbb N}$. We have
 \begin{enumerate}
 \item If $E_{k-1,k/n}$ is non-dicritical or $\rho_k = 1$, then
 \begin{equation*}
 \Bot(E_{k}) \leq \frac{\Top(E_{k})}{\rho_k}.
 \end{equation*}
 \item If $E_{k-1,k/n}$ is dicritical, then
 \begin{equation*}
 \Bot(E_{k}) \leq \frac{\Top(E_{k})}{\rho_k} + \frac{\rho_k-1}{\rho_k}.
 \end{equation*}
 or, what amounts to the same, $\Top(E_{k}) \geq \rho_k \Bot(E_{k}) - (\rho_k-1)$.
 \end{enumerate}
 Both inequalities are sharp.
\end{proposition}
Examples of equality are, for the non-dicritical case the algebraic equation $P\equiv y^n-x =0$
where $k=1$, $\rho_k=n$, $\Bot(E_{k})=1$ and $\Top(E_{k}) =n$, and for the
dicritical case $P\equiv py - n x y_1 =0$ in the differential setting, and
$P\equiv q^{p/n}y-y_1 =0$ in the $q$-algebraic one, where $\gcd (p,n)=1$, $k=p$, $\rho_k=n$
and $\Bot(E_{k})=\Top(E_{k}) =1$. {} 

\begin{proof}[Proof of Proposition \ref{pro:descent-step}]
Since $\Bot(E_{k}) \leq \Top(E_{k})$, the result is obvious if $\rho_k = 1$. 
Thus, we can assume that $k=e_l$ for some characteristic exponent $e_l$. 
For brevity, let $\Phi(C)=\Phi_{k-1,k/n}(C)$ be the initial polynomial of the element $E_{k-1,k/n}$. 
The argument hinges on Lemma \ref{lem:alpha-beta-descent}, but there are two cases.

\paragraph{\textbf{Non-dicritical case}} This means that $\Phi(C)\not \equiv 0$. Let
$\overline{h}$ be the multiplicity of $0$ as a root of $\Phi(C)$ ($\overline{h}$ may be $0$).
As $e_{l}$ is a characteristic exponent, Lemma \ref{lem:new-puiseux-exponent-new-side}
implies that $E_{k}$ is indeed a side of the Newton polygon $\mathcal{N}(P_{k})$, of
co-slope $k/n$. Write, as in Lemma \ref{lem:alpha-beta-descent},
$\Phi(C) = \alpha_{k}(C) + \delta_{k/n} C\beta_{k}(C)$. As $\Phi(C)\not \equiv 0$, it has degree less than or equal to $\Top(E_{k-1,k/n})=\Top(E_{k})$. 
 
Let $(\i, \Top(E_{k-1,k/n}))$ be the topmost vertex of $E_{k-1,k/n}$, 
which is also the topmost vertex of $E_{k}$ by Lemma \ref{lem:same-newton-polygon}.
We claim that any $(\iota^{\prime},j^{\prime})\in \mathcal{C}(P_{k -1}) \cap E_{k-1,k/n}$, with $j^{\prime}\leq \Top(E_{k-1,k/n})$, 
satisfies that
$s := \Top(E_{k-1,k/n}) - j^{\prime}$ is a multiple of
$r_l$.
Note that $E_{k-1,k/n}$ and $E_{k}$ have both co-slope $\frac{e_{l}}{n} = \frac{p_l}{r_1\cdots r_l}$ where $\gcd(p_l , r_l)=1$ by Equation \eqref{def:char}. We have
\begin{equation*}
 \iota^{\prime} = \iota + s \frac{k}{n} = \iota + s \frac{p_l}{r_1\cdots r_l} .
\end{equation*}
 As $(\iota^{\prime},j^{\prime})$ and $(\iota , \Top (E_{k-1,k/n}))$ belong to $L_{k/n} (P_{k-1})$, then
$\iota^{\prime}$ and $\iota$ belong to $\frac{\mathbb{Z}}{r_1\cdots r_{l-1}}$. It follows that 
\begin{equation*}
 s \frac{p_l}{r_l} = r_1 \hdots r_{l-1} (\iota^{\prime} - \iota) \in \mathbb{Z}
\end{equation*}
and thus $s$ is of the form $s=r_l\,r$ for some $r\in \mathbb{Z}_{\geq 0}$. As a consequence, $A^{k-1}_{\iota^{\prime} j^{\prime}}$ and 
$B^{k-1}_{\iota^{\prime} j^{\prime}}$ are $0$ except possibly when
$j^{\prime}=\Top(E_{k-1,k/n}) - r_{l}\,r$ for $r\in \mathbb{Z}_{\geq 0}$. Thus, $\Phi(C)$ can be
written as
 \begin{equation*}
 \Phi(C) = C^{\overline{h}}\overline{\Phi}(C^{r_l}),
 \end{equation*}
 that is, $\Phi(C)$ is, except for the factor $C^{\overline{h}}$, a polynomial in
 $C^{r_l}$. This implies that any root of $\Phi(C)$ different from $0$ has multiplicity
 at most $(\deg (\Phi(C)) - \overline{h})/r_l$. Let $m$ be the multiplicity of $a_{k}$
 as a root of $\Phi(C)$. Since $e_{l}$ is a characteristic exponent, $a_{k}\neq 0$, so
 that $m\leq (\deg (\Phi(C)) - \overline{h})/r_l$. Corollary
 \ref{cor:height-multiplicity} states that $m\geq \Bot(E_{k})$, as $E_{k-1, k/n}$ is
 non-dicritical. Thus, as $\Top(E_k)=\Top(E_{k-1,k/n})$, we obtain
 \begin{equation}\label{eq:bot-top-non-dicritical}
 \Bot(E_{k}) \leq m \leq \frac{\deg (\Phi(C)) - \overline{h}}{r_l} \leq
 \frac{\Top(E_{k})}{r_l} = \frac{\Top(E_{k})}{\rho_k},
 \end{equation}
 as desired.

 \paragraph{\textbf{Dicritical case}} 
 Denote by $(\iota_{j}, j)$ the point of ordinate $j$ in $L_{k/n} (P_{k-1})$.
 If\ $\Phi(C)\equiv 0 $, which is the dicritical condition, write $0=\Phi(C)=\alpha_{k}(C) + \delta_{k/n} C \beta_{k}(C)$. 
 Denote $b= \Bot (E_k)$. We have $A^{k}_{\iota_{b} b}+ \delta_{k/n} B^k_{\iota_{b} b}=0$ by Equation \eqref{eq:dicritical-residue}.
 Since $b= \Bot (E_k)$, at least one of them is non-vanishing and hence $A^{k}_{\iota_{b} b} \neq 0$ and $B^k_{\iota_{b} b} \neq 0$. 
 By definition, $B^{k}_{\iota_{j} j}(a_{k})=0$ for $j=1,\ldots, b-1$, so that by (\ref{eq:alpha-beta-descent-3}) in Lemma \ref{lem:alpha-beta-descent}
 or Equation \eqref{eq:alpha-beta-descent-3-aux}, we have $\beta_{k}^{(j-1)}(a_{k}) = 0$
 for $j=1,\ldots, b-1$. Moreover, we have $\beta_{k}^{(b-1)}(a_{k}) \neq 0$ since $B^k_{\iota_{b} b} \neq 0$. 
 Thus, $a_{k}$ is a root of multiplicity
 precisely $b-1$ of $\beta_{k}(C)$. 
 Let $\overline{h}\geq 1$ be the multiplicity of $0$ as a root of
 $C \beta_{k} (C)$. Considering that $r_{l}$ is a novel factor of the denominator of
 $k/n$, the same argument as in the previous case shows that both
 $\alpha_{k}(C)/C^{\overline{h}}$ and $C\beta_{k}(C)/C^{\overline{h}}$ are, indeed,
 polynomials in $C^{r_l}$. On the other hand, dicriticalness also gives that
 $\beta_{k}(C)$ has degree equal to $\Top(E_{k})-1$. Thus, we get
 \begin{equation*}
 \Bot(E_{k})-1 = b-1 \leq \frac{\deg (C \beta_{k}(C)) - \overline{h}}{r_l} =
 \frac{\Top(E_{k})-\overline{h}}{r_l},
 \end{equation*}
 from which follows that
 \begin{equation*}
 \Bot(E_{k}) \leq \frac{\Top(E_{k})}{r_l} + \frac{r_l-\overline{h}}{r_l} \leq
 \frac{\Top(E_{k})}{\rho_k} + \frac{\rho_k-1}{\rho_k},
 \end{equation*}
 as desired. The last inequality is an equality if $\overline{h}=1$, i.e. if $\beta_{k} (0) \neq 0$.
\end{proof} 
\begin{corollary}\label{cor:inequalities-T-b-sharp}
If $E_{k-1, k/n}$ is non-dicritical and the equality $ \Bot(E_{k}) = \Top(E_{k}) / \rho_k$ holds, then 
$\Phi_{k-1,k/n}(C)$ is of the form 
$\Phi_{k-1,k/n}(C) = u (C^{\rho_{k}} - a_{k}^{\rho_{k}})^ {\Bot(E_{k})}$ for some $u \in {\mathbb C}^{*}$.
\end{corollary}
\begin{proof}
 Assume first $\rho_k=1$. Since the degree of $\Phi_{k-1,k/n}(C)$ is at most
 $\Top(E_{k})$ and $a_k$ is a root of $\Phi_{k-1,k/n}(C)$ of multiplicity at least
 $\Bot(E_{k})$ (Corollary \ref{cor:height-multiplicity}), it follows that
 $\Phi_{k-1,k/n}(C) = u (C - a_{k})^ {\Bot(E_{k})}$ for some $u \in {\mathbb C}^{*}$.
 
 Assume now $\rho_{k}>1$ and consider the notation in the proof of Proposition
 \ref{pro:descent-step}. The equality $ \Bot(E_{k}) = \Top(E_{k}) / \rho_k$ implies by
 \eqref{eq:bot-top-non-dicritical}, that $\overline{h}=0$ and
 $\deg (\Phi (C)) = \Top(E_{k})$. As a consequence, $\Phi (C)$ is a polynomial in
 $C^{\rho_k}$ with a non-vanishing root $a_k$ of multiplicity $\deg (\Phi
 (C))/\rho_k$. Thus $\Phi (C) = u (C^{\rho_{k}} - a_{k}^{\rho_{k}})^ {\Bot(E_{k})}$ for
 some $u \in {\mathbb C}^{*}$.
\end{proof} 

\subsection{Proof of Theorem \ref{the:main}}
 Denote $\mu=\ord(s(x))$ and let $k_0=n\, \mu$.
 Since the coefficients $a_i$ of $s(x)$ are zero for
 $i=1,\ldots,k_0-1$, we have that $P=P_{k_0-1}$ 
 and
 \begin{equation*}
 H(P)\geq H(P,s(x))=\Top(E_{P,\mu})=\Top(E_{k_0-1,k_0/n})=\Top(E_{k_0}).
\end{equation*}
 We split the proof in two cases. 
 
 Assume $d=0$, that is, there are no dicritical characteristic exponents. We obtain
 $\Top(E_{k_0}) \geq \prod_{j=1}^{g} r_j \geq 1$ by applying Proposition
 \ref{pro:descent-step} iteratively from $j\geq k_0$, and the fact that
 $\Top(E_{j,(j+1)/n})\leq \Bot(E_{j,j/n})$ for any $j$, as in Lemma
 \ref{lem:top-less-bottom}.
 
 Now suppose $d \geq 1$. We get
 $\Top(E_{k_0})\geq r_1 \cdots r_{i_1-1} \Top(E_{e_{i_1}})$ by successive applications
 of the non-dicritical case of Proposition \ref{pro:descent-step}. Using again
 Proposition \ref{pro:descent-step}, now in the dicritical case, we get
 \begin{equation}\label{eq:repeated-dicr}
 r_1 \cdots r_{i_1-1} \Top(E_{e_{i_1}}) \geq
 r_1 \cdots r_{i_1} \Bot(E_{e_{i_1}}) - (r_1 \cdots r_{i_1}- r_1 \cdots r_{i_1-1})
 \end{equation}
 which gives, taking into account that $\Bot(E_{e_{i_1}}) \geq \Top(E_{e_{i_1} +1}) $,
 \begin{equation*}
 \Top(E_{k_0}) \geq
 r_1 \cdots r_{i_1} \Top(E_{e_{i_1} +1 }) - (r_1 \cdots r_{i_1}- r_1 \cdots r_{i_1-1}). 
 \end{equation*}
 We obtain the desired bound
 \begin{equation*}
 \Top(E_{k_0}) \geq \prod_{j=1}^g r_j - 
 \sum_{k=1}^d \left(\prod_{j=1}^{i_k} r_{j} - \prod_{j=1}^{i_k-1}r_j
 \right)
 \end{equation*}
 by iterating this argument. 

 \subsection{Bound without assumptions on dicritical exponents}We shall see next that
dicritical elements cannot be immediately consecutive, which will provide us with a bound in which the dicritical elements play no role. 
\begin{lemma}\label{lem:no-consecutive-dicritical}
 Let $Q=(\kappa, b)$ be the bottom point of the element $E_k$. 
 Assume $E_{k-1,k/n}$ is dicritical. Then the point $Q$ does not
 belong to any later dicritical element. As a consequence, if $k$ is a
 characteristic exponent, then the next characteristic exponent $e_{\ell}$ is either
 non-dicritical or the element $E_{e_{\ell}}$ satisfies
 \begin{equation*}
 \Top(E_{e_{\ell}}) \leq b -1.
 \end{equation*}
 \end{lemma}
\begin{proof}
 Let $A^{k}_{\kappa b},B^{k}_{\kappa b}$ be the coefficients of
 $P_k=A^{k}(x,y)+B^{k}(x,y)y_{1}$ corresponding to the point $Q$. As $E_{k-1,k/n}$ is
 dicritical,$\Phi_{k-1,k/n}(C)\equiv 0$, and by Equation \eqref{eq:dicritical-residue} applied to $j=b$ and $C=a_{k}$,
 \begin{equation*}
 A^{k}_{\kappa b} + \delta_{k/n} B^{k}_{\kappa b} = 0,
 \end{equation*}
 so that $\delta_{k/n} = -A^{k}_{\kappa b}/B^{k}_{\kappa b}$. We claim that $Q$ does not
 belong to a dicritical element $E_{m-1,m/n}$ with $m > k$. By Lemma
 \ref{lem:same-newton-polygon} and $b \geq \Top (E_{k,(k+1)/n})$, 
 $A^{k}_{\kappa b}=A^m_{\kappa b}$ and
 $B^{k}_{\kappa b}=B^m_{\kappa b}$ for $m>k$. If $E_{m-1,m/n}$ were dicritical, we should
 have
 \begin{equation*}
 A^{k}_{\kappa b} + \delta_{m/n}B^{k}_{\kappa b} = 0, 
 \end{equation*} 
 which is a contradiction. Here we use the condition $|q| \neq 1$ in the case of
 $q$-difference equations, so that $q^{m/n}\neq q^{k/n}$. The result follows now
 straightforwardly.
\end{proof}
The first consequence of this Lemma is:
\begin{corollary}\label{cor:two-consecutive-dicritical}
 if $e_\ell,e_{\ell+1}$ are two consecutive dicritical characteristic exponents, then
 \begin{equation}
 \Top(E_{e_{\ell+1}}) \leq \Bot(E_{e_\ell})-1 < \frac{\Top(E_{e_\ell})}{r_\ell}.
 \end{equation}
 As a consequence, if {} $e_\ell,e_{\ell+1},\ldots, e_{\ell+p}$ {} is a sequence of consecutive dicritical characteristic exponents with {} $p \geq 1$, {} then
 \begin{equation}\label{eq:several-consecutive-dicritical}
 {} \Top(E_{e_{\ell+p}}) < \frac{\Top(E_{e_{\ell}})}{r_\ell\ldots r_{\ell+p-1}}. {}
 \end{equation} 
\end{corollary}
We can proceed now to prove Corollary \ref{cor:a}.
\begin{proof}[Proof of Corollary \ref{cor:a}]
 In Theorem \ref{the:main}, we can improve the inequality as follows: we say that
 $e_\ell$ is a \emph{terminally dicritical} characteristic exponent if $e_\ell$ is
 dicritical and either $\ell=g$ or $e_{\ell+1}$ is non-dicritical. By Equation
 \eqref{eq:several-consecutive-dicritical} in Corollary
 \ref{cor:two-consecutive-dicritical}, the argument giving Equation
 \eqref{eq:repeated-dicr} in the proof of Theorem \ref{the:main} can be restricted
 to terminally dicritical exponents. Hence, if $\ell_1,\ldots \ell_s$ is the sequence of
 indices of terminally dicritical exponents,
 \begin{multline}\label{eq:large-ineq}
 \Top(E_{k_0}) \geq \prod_{j=1}^g r_j - \sum_{k=1}^{s} \left( \prod_{j=1}^{\ell_k} r_j
 - \prod_{j=1}^{\ell_k-1} r_j
 \right) \\
 =\left( \prod_{j=1}^g r_j - \prod_{j=1}^{\ell_{s}} r_j \right)+ \sum_{k=2}^{s}
 \left( \prod_{j=1}^{\ell_k -1} r_j - \prod_{j=1}^{\ell_{k-1}} r_j \right) +
 \prod_{j=1}^{\ell_1 -1} r_j,
 \end{multline}
 where all the terms in parentheses are non-negative.

 If $\ell_{s}<g$, i.e. the last characteristic exponent is non-dicritical, then
 \begin{equation*}
 \Top(E_{k_0}) \geq
 \prod_{j=1}^g r_j - \prod_{j=1}^{\ell_s} r_j \geq
 \prod_{j=1}^g r_j- \prod_{j=1}^{g-1}r_j\geq
 \prod_{j=1}^{g-1}r_j>
 \prod_{j=1}^{g-1}r_j - \prod_{j=1}^{g-2}r_j
 \end{equation*}
 because $r_{g}\geq 2$. Otherwise, if $\ell_s=g$ and $s=1$, then \eqref{eq:large-ineq}
 becomes
 \begin{equation*}
 \Top(E_{k_0}) \geq \prod_{j=1}^gr_j - \left(
 \prod_{j=1}^gr_j - \prod_{j=1}^{g-1}r_j
 \right) = \prod_{j=1}^{g-1}r_j >
 \prod_{j=1}^{g-1}r_j - \prod_{j=1}^{g-2} r_j.
 \end{equation*}
 When $\ell_s=g$ and $s>1$ then, we reason as follows: all
 the terms in the inner summation in \eqref{eq:large-ineq} are nonnegative, and the
 largest one is the one with $k=s$ because $r_j\geq 2$. Thus, by keeping just this
 term, we obtain
 \begin{align*}
 \left( \prod_{j=1}^g r_j - \prod_{j=1}^{\ell_{s}} r_j \right)+ \sum_{k=2}^{s}
 \left( \prod_{j=1}^{\ell_k -1} r_j - \prod_{j=1}^{\ell_{k-1}} r_j \right) +
 \prod_{j=1}^{\ell_1 -1} r_j\geq \\
 \left(\prod_{j=1}^g r_j - \prod_{j=1}^{g} r_j\right) +
 \prod_{j=1}^{g-1} r_j - \prod_{j=1}^{\ell_{s-1}}r_j + \prod_{j=1}^{\ell_1-1}r_j
 >\\
 \prod_{j=1}^{g-1}r_j- \prod_{j=1}^{g-2}r_j
 \end{align*}
 because $\ell_{s-1}\leq g-2$.
 \end{proof}

Corollary \ref{cor:a} is the best we can say in general for any kind of covered equation. However, for \emph{generic} or \emph{contracting} 
$q$-difference equations and for general differential equations, we can be more precise. The genericity condition for $q$-difference equations we shall state becomes clear after Lemmas \ref{lem:quotient-at-b-from-quotient-at-T} and \ref{lem:improper} below.

\subsection{Reasonable equations: relations between consecutive dicritical elements}
As we explained above, the dicritical property does not affect only the initial polynomial $\Phi_{k-1,k/n}(C)$ but it also
creates relations between the coefficients $A^{k}_{\iota j}$ and $B^{k}_{\iota j}$ falling
on a point $(\iota, j)$ in the dicritical element $E_{k}$. These relations, which in some
sense bring to mind the concept of \emph{residue} or \emph{index} for a singular
holomorphic foliation along a non-singular separatrix, will be key to discern what kind of
covered equations admit an even sharper bound of $H(P)$ in terms of the factors
$r_1,\ldots, r_g$. Hence the name given to the following concept.

\begin{definition}\label{def:residue}
 Consider the element $E_{k}$ with topmost vertex $(\iota, t)$ and lowest vertex $(\kappa, b)$. Using \eqref{eq:initial-form}, 
 we define the $k$-th \emph{top} and \emph{bottom} \emph{residues} as
 \begin{equation*}
 \tres_{k} = \frac{A^{k}_{\iota t}(0)}{B^{k}_{\iota t}(0)} =
 \frac{A^{k-1}_{\iota t}}{B^{k-1}_{\iota t}},
 \,\,\,
 \bres_{k} = \frac{A^{k}_{\kappa b}(a_{k})}{B^{k}_{\kappa b}(a_{k})}
 \end{equation*}
 respectively. By convention, $\ast/0=\infty$ (as $0/0$ does not happen by definition).
\end{definition}
The following result gives a necessary condition for the inequality in case (1) of Proposition \ref{pro:descent-step} to be an equality, which is the worst case for our bounds.
\begin{lemma}\label{lem:quotient-at-b-from-quotient-at-T}
 Let $\tres_{k}$ and $\bres_{k}$ be the $k$-th top and bottom residues, respectively. Assume that the initial polynomial $\Phi(C)=\Phi_{k-1,k/n}(C)$ is of the form $\Phi(C)=u(C^{\rho_k}-a_k^{\rho_k})^{\Bot (E_k)}$ where $\Phi(C)$ has degree $\Top (E_k)$, $\rho_k$ is as in Definition \ref{def:factor-of-exponent}, and $u$ is a non-zero constant (as a consequence, $E_{k-1,k/n}$ is non-dicritical). Then
 \begin{equation*}
 \bres_{k}= \rho_k\tres_{k} + (\rho_k-1) \delta_{k/n}
 \end{equation*}
if $\tres_{k} \neq \infty$, or $\bres_{k}=\infty$ otherwise. 
\end{lemma} 
\begin{proof}
 Let $(\iota,t)$ be the top point of $E_k$ and $(\kappa, b)$ its bottom
 one. Before proceeding, notice that
 $u=A^{k}_{\iota t}(0)+\delta_{k/n}B^{k}_{\iota t}(0)$. Assume first $\rho_k=1$. Our
 hypothesis implies $t=b$ and hence $\bres_{k} = \tres_{k}$ as desired. From now on we
 assume $\rho_k>1$ and hence $a_k \neq 0$.
 
 Analogously to the proof of the non-dicritical case of Proposition \ref{pro:descent-step}, we obtain that 
 all the points in $E_{k-1,k/n}$ are of the form $(\iota+s\rho_{k} k/n,t-s\rho_{k})$ for
 some $s\in \mathbb{Z}_{\geq 0}$. Moreover, since $\Phi(0) \neq 0$, it follows that $t$ is of the form $t=m \rho_k$ for some $m \in \mathbb{Z}_{\geq 0}$.
 Thus, if $(\iota,j)$ is a point in $E_{k-1,k/n}$, then
 $j=t -s\rho_{k}=(m-s)\rho_{k}$ for some $s\geq 0$. Thus, $B^{k-1}_{\iota_{j},j}\neq 0$, where $\iota_{j}=\iota + (m\rho_{k}-j)k/n$, implies that $j$ is
 a multiple of $\rho_k$. Therefore, the polynomial $C\beta_{k}(C)$ is in fact a
 polynomial in $C^{\rho_k}$ of degree at most $t$. By Corollary
 \ref{cor:height-multiplicity},
 as $a_k$ is a root of multiplicity at least $b-1$ of $\beta_k (C)$,
 then either $\beta_k (C)=0$ or $C\beta_k (C)=vC^{\rho_k}(C^{\rho_k}-a_k^{\rho_k})^{b-1}$
 for $v=B^{k}_{\iota t}(0)$. In the first case $\tres_{k}=\infty$ and certainly
 $\bres_{k}=\infty$ as well. From now on we assume
 $\beta_k (C)=v C^{\rho_k-1}(C^{\rho_k}-a_k^{\rho_k})^{b-1}$ with $v \neq 0$. By
 \eqref{eq:alpha-beta-descent-3} in Lemma \ref{lem:alpha-beta-descent}, we obtain
 \begin{equation*}
 B^{k}_{\kappa b}(a_k) = \frac{1}{(b-1)!}\beta_k ^{(b-1)}(a_k)
 \end{equation*}
 which gives, for $\xi$ a primitive $\rho_k$-th root of unity:
 \begin{equation*}
 B^{k}_{\kappa b}(a_k) = \frac{B^{k}_{\iota t}(0)}{(b-1)!} a_{k}^{\rho_k-1} (b-1)!\prod_{j=1}^{\rho_k-1}\left(a_k - \xi^ja_k\right)^{b-1} = B^{k}_{\iota t}(0) a_{k}^{(\rho_k-1)b} \prod_{j=1}^{\rho_k-1}\left(1 - \xi^j\right)^{b-1}.
 \end{equation*}
 On the other hand, for the same $\xi$, we have 
 \begin{equation*}
 \frac{1}{b!}\Phi^{(b)}(a_k) = u \prod_{j=1}^{\rho_k-1}
 \left(a_k - \xi^{j}a_k \right)^b = u a_{k}^{(\rho_k-1)b} \prod_{j=1}^{\rho_k-1}
 \left( 1 - \xi^j \right)^b.
 \end{equation*}
 Applying Equations \eqref{eq:alpha-beta-descent-2} and \eqref{eq:alpha-beta-descent-3} of
 Lemma~\ref{lem:alpha-beta-descent}, and the two equalities above, we get
 \begin{equation*}
 \frac{A^{k}_{\kappa b}(a_k)}{B^{k}_{\kappa b}(a_{k})}=\frac{\Phi^{(b)}(a_k)}{b!B^{k}_{\kappa b}(a_k)} - \delta_{k/n} =
 \frac{u a_{k}^{(\rho_k-1)b} \prod_{j=1}^{\rho_k-1}
 \left( 1 - \xi^j \right)^b}
 {B^{k}_{\iota t}(0) a_{k}^{(\rho_k-1)b} \prod_{j=1}^{\rho_k-1}\left(1 - \xi^j\right)^{b-1}} -
 \delta_{k/n} .
 \end{equation*}
 And, as $u=A^{k}_{\iota t}(0)+\delta_{k/n}B^{k}_{\iota t}(0)$, we obtain
 \begin{equation*}
 \bres_{k}=\frac{A^{k}_{\kappa b}(a_k)}{B^{k}_{\kappa b}(a_k)} = \frac{A^{k}_{\iota t}(0)+\delta_{k/n} B^{k}_{\iota t}(0)}{B^{k}_{\iota t}(0)} \prod_{j=1}^{\rho_k-1}(1-\xi^j) - \delta_{k/n} . 
 \end{equation*}
 But since $\prod_{j=1}^{\rho_k-1} (1-\xi^j) = (C^{\rho_k}- 1)' (1)$, we have $\prod_{j=1}^{\rho_k-1} (1-\xi^j) =\rho_k$, so that 
 \begin{equation*}
 \bres_{k} = \rho_{k}\tres_{k} + (\rho_k-1)\delta_{k/n}
 \end{equation*}
 as desired. 
\end{proof}
\begin{corollary}\label{cor:residue-in-pure-chain}
Assume that either $E_{k-1,k/n}$ is dicritical or $\Bot(E_{k})=\Top(E_{k})/\rho_k$. Then
 \begin{equation}
 \label{equ:recurrence}
 \bres_{k} = \rho_{k} \tres_{k} + (\rho_{k}-1) \delta_{k/n}.
 \end{equation} 
 Furthermore:
 \begin{enumerate}
 \item If $P$ is a differential equation and $\tres_{k}$ is a real number with $\tres_{k} \geq -k/n$, then $\bres_{k} \geq -k/n$.
 \item If $P$ is a $q$-difference equation with $q\in \mathbb{R}^{+} \setminus \{ 1 \}$ then $\tres_{k}\geq -q^{k/n}$ (resp. $\tres_{k}\leq -q^{k/n}$) if and only if 
 $\bres_{k} \geq -q^{k/n}$ (resp. $\bres_{k} \leq -q^{k/n}$). 
 \end{enumerate}
\end{corollary}
\begin{proof}
 If $E_{k-1,k/n}$ is dicritical, then all the coefficients of $\Phi_{k-1,k/n}(C)=\alpha_k(C)+\delta_{k/n}C\beta_k(C)$ are $0$ and Equation \eqref{eq:dicritical-residue} gives:
 \begin{equation*}
 A^{k}_{\iota_{j} j} = - \delta_{k/n} B^{k}_{\iota_{j} j}
 \end{equation*}
 for each $(\iota_{j},j)\in E_{k}$, and $j=0,\ldots, \Top(E_{k-1,k/n})$, so that
 \eqref{equ:recurrence} is trivial in this case with $\tres{k}=\bres_k=-\delta_{k/n}$. When $E_{k-1,k/n}$ is not
 dicritical, the result follows from the fact that $\Bot(E_{k})=\Top(E_{k})$ if $\rho_{k}=1$, and
 from 
 Lemma
 \ref{lem:quotient-at-b-from-quotient-at-T} if $\rho_{k} >1$. The other results follow
 trivially from that equality.
\end{proof}
Our next goal is to improve the lower bound provided by Corollary \ref{cor:a}. Indeed, we can obtain a better lower bound in the case of differential equations and, for $q$-difference equations, when $q$ is not a root of a certain subset of algebraic equations, the so called unreasonable equations, which we shall describe below. 
\begin{definition} 
A polynomial $Q(z)= z^m + u_{m-1} z^{m-1} + \hdots + u_{0} \in {\mathbb R}[z]$ is \emph{improper} if $u_0, \hdots, u_{m-1} \in {\mathbb R}^{+}$ and 
$1 \leq u_{m-1} \leq u_{m-2} \leq \hdots \leq u_0$.
\end{definition}
\begin{lemma}
\label{lem:improper}
Any complex root $z_0$ of an improper polynomial $Q(z)$ satisfies $|z_0| \geq 1$ and 
$|z_0| \leq \max \left( \frac{u_0}{u_1}, \hdots, \frac{u_{m-2}}{u_{m-1}}, \frac{u_{m-1}}{1} \right)$.
\end{lemma}
\begin{proof}
It is obvious that no positive real number can be a root of $Q(z)$. Denote $u_m = 1$. We have
\[ (z-1) Q(z) = (z^{m+1} - z^{m}) + \sum_{j=0}^{ m-1} u_j (z^{j+1} - z^{j}) = z^{m+1} + \sum_{j=1}^{m} (u_{j-1} -u_j ) z^{j} - u_0. \]
Given $z_0 \in {\mathbb C}$ with $|z_0| < 1$, we obtain
\[ \left| z_{0}^{m+1} + \sum_{j=1}^{m} (u_{j-1} -u_j ) z_{0}^{j} \right| < 1 + (u_{m-1} - u_m) + \hdots + (u_0 - u_1) = u_0, \]
so $z_0$ is not a root of $Q(z)$. 

Let us prove the other inequality. Denote $r = \max_{0 \leq j < m-1} (u_{j}/u_{j+1})$. It suffices to prove that $\tilde{Q}(z):= z^{m} Q(r/z)/u_0$ has no roots of modulus less than $1$.
Since $\tilde{Q}(z)$ is improper, the result is a consequence of the first part of the proof.
\end{proof}

Corollary \ref{cor:residue-in-pure-chain} on top and bottom residues has an important
consequence on chains of non-dicritical elements that follow a dicritical one. We
want to apply Proposition \ref{pro:descent-step} repeatedly without a priori information on the
dicritical exponents in order to obtain a sharp bound for $H(P,s(x))$, and hence $H(P)$ (Theorem \ref{teo:reasonable}). 
To this end, it is important to study situations in which the iterative use of 
Proposition \ref{pro:descent-step} and Lemma \ref{lem:top-less-bottom}
gives us poor lower bounds.
This is the goal of the following result in which we analyze a ``worst case scenario" for a pair of 
dicritical elements $E_{k-1, k/n}$ and $E_{l, (l+1)/n}$
that are separated by a chain of non-dicritical elements $E_{k, (k+1)/n}, E_{k+1, (k+2)/n}, \hdots, E_{l-1, l/n}$. 
We shall also see 
that it can only happen in very special cases of $q$-difference equations. This
motivates Definition \ref{def:reasonable}, of reasonable equations. 
\begin{proposition}\label{pro:no-pure-dicritical-chain}
 Assume that $E_{k-1,k/n}$ is a dicritical element and that for
 $j=k+1,\ldots, l$, any element $E_{j}$ satisfies
 $\Bot(E_{j })=\Top(E_{j})/\rho_j$ and
 $\Top(E_{j})=\Bot(E_{j-1})$. Assume also that if $\rho_j>1$
 then $E_{j-1,j/n}$ is non-dicritical for $j=k+1,\ldots, l$. If the
 bottom of $E_{l}$ belongs to a dicritical element, then the
 equation $P$ is $q$-differential and, denoting $s=q^{1/n}$, we have
 \begin{equation}
 \label{equ:unreasonable}
 s^{\overline{l}-k} + \rho_{\overline{l}}\, s^{\overline{l}-k-1} + \rho_{\overline{l}-1}\,
 \rho_{\overline{l}}\, s^{\overline{l}-k-2} + \hdots +
 \rho_{k+1} \hdots \rho_{\overline{l}} =0 
 \end{equation}
 for some $\overline{l} \geq l$, which implies that $q^{1/n}$ is the root of an improper polynomial.
\end{proposition}
\begin{proof}
 By hypothesis, the bottom of $E_l$ belongs to a first dicritical element,
 $E_{\overline{l},(\overline{l}+1)/n}$ with $\overline{l}\geq l$, so that
 $\Bot(E_{l})=\Top(E_{\overline{l},(\overline{l} +1)/n})=\Top(E_{\overline{l}+1})$. This implies that, 
 we have $\Bot(E_{j})=\Top(E_{j })$ for any $l<j\leq \overline{l}$ and 
 $\Top(E_{j})=\Bot(E_{j-1})$ for any $l<j\leq \overline{l} +1$. Moreover, $\Bot(E_{j})=\Top(E_{j })$
 implies $\rho_j=1$ for any $l<j\leq \overline{l}$ by the first case of Proposition \ref{pro:descent-step}. 
 This property, together with the hypotheses, provide $\Bot(E_{j })=\Top(E_{j})/\rho_j$ for any $k < j \leq \overline{l}$.

 If $P$ is a differential equation, then the top and bottom residues of $E_{k}$ satisfy
 $\tres_{k}=\bres_{k}=-k/n$ because of the dicritical condition and Equation \eqref{eq:dicritical-residue}. 
 Now, $\tres_{k} \geq -k/n$ implies
 $\tres_{k+1} = \bres_{k} \geq -k/n > -(k+1)/n$,
 where the first equality is a consequence of $\Top (E_{k+1}) = \Bot(E_{k})$.
 An iterative application of
 Corollary \ref{cor:residue-in-pure-chain} provides the inequality
 $\bres_{j} \geq -j/n$ for $j=k ,\ldots, \overline{l}$. This prevents the bottom of each
 $E_{j}$, $k \leq j \leq \overline{l}$, from being dicritical at some later step, as
 $\bres_{j}$ should be $-m/n$ for some $m>j$ otherwise. Thus, $P$ cannot be a
 differential equation.

 We can then assume that $P$ is a $q$-difference equation. Denoting $s=q^{1/n}$, let us
 show, by induction, that
 \begin{equation}
 \label{equ:recurrence_l} 
 \bres_{j} = s^{k} (s-1)(s^{j-k} +\rho_{j} s^{j-k-1} + \rho_{j-1} \rho_{j} s^{j-k-2}
 + \hdots + \rho_{k+1} \hdots \rho_j ) - s^{j+1} 
 \end{equation}
 for any $k \leq j \leq \overline{l}$, from which the result follows because 
 $\bres_{\overline{l}}=-s^{\overline{l}+1}$
 by $\Top(E_{\overline{l}+1})=\Bot(E_{\overline{l}})$ and the dicriticalness of $E_{\overline{l},(\overline{l}+1)/n}$.
 Equation \eqref{equ:recurrence_l} reads $\bres_{k}=s^k(s-1)-s^{k+1} = -s^k$ in the base case $j=k$. In this
 case, as $E_{k-1,k/n}$ is dicritical, we get 
 \[ \tres_{k} =\bres_{k} = -\delta_{k/n}=-s^{k} \] 
 by \eqref{eq:dicritical-residue}, so that Equation (\ref{equ:recurrence_l}) holds indeed
 for $j=k$. Suppose Equation (\ref{equ:recurrence_l}) holds for some
 $j < \overline{l}$. We have $\tres_{j+1} = \bres_{j}$ because $\Top(E_{j+1})=\Bot(E_j)$.
 Since $\Bot(E_{j+1 })=\Top(E_{j+1})/\rho_{j+1}$, 
 Corollary \ref{cor:residue-in-pure-chain} gives the equality
 \[
 \bres_{j+1} = \rho_{j+1} \tres_{j+1} + (\rho_{j+1}-1) s^{j+1} .
 \] Thus,
 \[
 \bres_{j+1} = \rho_{j+1} \bres_j + (\rho_{j+1}-1)s^{j+1},
 \]
 so that the induction hypothesis on $j$ gives, then, by direct substitution
 \begin{equation*}
 \bres_{j+1} = \rho_{j+1}\left(s^k(s-1)(s^{j-k}+\rho_js^{j-k-1}+\cdots
 +\rho_{k+1}\cdots\rho_j) - s^{j+1}\right) + (\rho_{j+1}-1)s^{j+1}.
 \end{equation*}
 Inserting the common factor $\rho_{j+1}$ into the parenthesis starting with $s^{j-1}$,
 we get
 \begin{multline*}
 \bres_{j+1} = 
 s^k(s-1)\left(\rho_{j+1}s^{j-k} + \rho_{j}\rho_{j+1}s^{j-k-1} + \cdots +
 \rho_{k+1}\cdots \rho_{j+1}\right) \\
 -\rho_{j+1}s^{j+1} + \rho_{j+1}s^{j+1}-s^{j+1}=\\
 s^k(s-1)\left(\rho_{j+1}s^{j-k} + \rho_{j}\rho_{j+1}s^{j-k-1} + \cdots +
 \rho_{k+1}\cdots \rho_{j+1}\right) -s^{j+1}.
 \end{multline*}
 Finally, introducing the zero expression $s^{j+1-k} - s^{j+1-k}$ into the second parenthesis and operating the second term, we get
 \begin{multline*}
 \bres_{j+1} = 
 s^k(s-1)\left(s^{j+1-k}+\rho_{j+1}s^{j-k} + \rho_{j}\rho_{j+1}s^{j-k-1} + \cdots +
 \rho_{k+1}\cdots \rho_{j+1}\right) - \\
 s^{j+2} +s^{j+1} - s^{j+1},
 \end{multline*}
 which is Equation (\ref{equ:recurrence_l}) for $j+1$. 
 \end{proof}
\begin{definition}\label{def:reasonable}
 We say that $P$ is a \emph{reasonable} equation if it is a differential equation or if it is a $q$-difference equation such that $q^{1/n}$ is not a solution of any 
 equation of the form (\ref{equ:unreasonable}).
 \end{definition}
 \begin{remark} \label{rem:reasonable}
 The following cases provide reasonable $q$-difference equations, noticing that $q\in \mathbb{R}^+$ does not imply $q^{1/n}\in \mathbb{R}^+$, as one might have chosen $\log(q)$ to be non-real:
 \begin{itemize}
 \item If $q$ is a positive real number, different from $1$, and $q^{1/n} \in \mathbb{R}^{+}$,
 \item or $|q| < 1$, 
 \item or $|q|^{\frac{1}{n}} > \max (r_1, \hdots, r_g)$, 
 \item or $q$ is a trascendental number over $\mathbb{Q}$.
 \end{itemize}
 This is a direct consequence of Proposition \ref{pro:no-pure-dicritical-chain} and Lemma
 \ref{lem:improper} since a non-reasonable equation only happens if $q^{1/n}$ is a root
 of an improper polynomial. So in order for a $q$-difference equation to be reasonable,
 $q$ just needs to avoid a countable subset of the annulus
 $\{ 1 < |q|^{1/n} \leq \max (r_1, \hdots, r_g) \}$.
\end{remark}

The next result is what makes reasonable equations interesting: whenever there are two dicritical elements, any ``fastidious'' term $(\rho_k-1)/\rho_k$ between them coming from item (2) of Proposition \ref{pro:descent-step} disappears. 

\begin{lemma}\label{lem:inter-dicritical}
 Let $P$ be a $1$-covered reasonable equation. 
 Assume the following conditions hold:
 \begin{enumerate}
 \item The elements $E_{k-1,k/n}$ and $E_{l-1,l/n}$ are dicritical, and 
 \item Any characteristic exponent $e_j$ with $k<e_j<l$ is non-dicritical.
 \end{enumerate}
 Then
 \begin{equation*}
 \Top(E_{l}) < \frac{\Top(E_{k})}{\rho_{k}\cdots \rho_{l-1}}. 
 \end{equation*}
\end{lemma}
\begin{proof}
 If $\Top(E_j)<\Bot(E_{j-1})$ for some $j\in \left\{ k+1,\ldots, l \right\}$ then, applying Proposition \ref{pro:descent-step} iteratively, we obtain:
 \begin{equation*}
 \Top(E_{j})\leq \Bot(E_{j-1}) -1 \leq \frac{\Bot (E_k)}{\rho_{k+1}\cdots \rho_{j-1}} - 1 
 \leq \frac{\frac{\Top(E_{k})}{\rho_k} + \frac{\rho_k-1}{\rho_{k}}}{\rho_{k+1}\cdots \rho_{j-1}} - 1 
 \end{equation*}
 which gives
 \begin{equation*}
 \Top(E_j) \leq \frac{\Top(E_{k})}{\rho_k\cdots \rho_{j-1}} + \frac{\rho_k-1}{\rho_k \cdots \rho_{j-1}} - 1 < \frac{\Top(E_{k})}{\rho_k\cdots \rho_{j-1}}, 
 \end{equation*}
 and as all characteristic exponents strictly between $j$ and $l$ are non-dicritical, Proposition \ref{pro:descent-step} 
 gives once more the result.

 If $\Top(E_j)=\Bot(E_{j-1})$ for all $j\in \left\{ k +1, \ldots, l \right\}$, there are two possibilities:
 \begin{itemize}
 \item If $\Bot(E_j)< \frac{\Top(E_j)}{\rho_j}$ for some $j$ with $k < j <l$, we get $\Bot(E_j)\leq \frac{\Top(E_j)-1}{\rho_j}$ and 
 the same argument as above gives:
 \begin{equation*}
 \Top(E_{j+1}) \leq \Bot(E_j) \leq \frac{\Top(E_j)-1}{\rho_j} 
 \leq \frac{\Top(E_{k})}{\rho_{k}\cdots \rho_j} + \frac{\rho_k-1}{\rho_k\cdots \rho_{j}} - \frac{1}{\rho_j} 
 < \frac{\Top(E_{k})}{\rho_k\cdots \rho_{j}} 
 \end{equation*}
 and as the characteristic exponents between $j+1$ and $l$ are non-dicritical, the result follows.
 \item If $\Bot(E_j)= \Top(E_j)/\rho_j$ for all $j\in \left\{ k+1, \ldots, l-1 \right\}$, then Proposition \ref{pro:no-pure-dicritical-chain} implies that $\Top(E_{l})$ cannot be equal to $\Bot(E_{l-1})$, so that this case cannot happen, and we are done.
 \end{itemize}
\end{proof}

\subsection{Proof of Theorem \ref{teo:reasonable}} 
In order to prove Theorem \ref{teo:reasonable},
we need to be able to control the elements of $P$ corresponding to exponents $k/n$ for $k/n<\ord(s(x))$: despite their coefficients being $0$, the elements $E_{k-1,k/n}$ could, in principle, be dicritical, preventing the desired bound from holding. This is tackled in the next technical result, from which Theorem \ref{teo:reasonable} follows trivially, setting $m=\ord(s(x))n$. 
\begin{proposition} \label{pro:b}
 Let $s(x)$ be a Puiseux solution of genus $g$ of the reasonable equation $P=0$ and fix $1 \leq m \leq \ord (s(x)) n$. Then
\begin{equation}
 \label{equ:ineq}
\Top (E_{P,m/n}) \geq r_1\cdots r_{g-1} .
\end{equation} 
The inequality is strict unless, possibly,
if all the following conditions hold: $e_{g}/n$ is the unique
dicritical exponent of $s(x)$ in $\mathbb{Z}_{\geq m} / n$,
$\Bot(E_{e_{g}})=\Top(E_{e_g})=1$, and for any $m \leq j <e_g$, 
one has $\Bot(E_{j})=\Top(E_{j})/\rho_j$ and
$\Bot(E_{j})=\Top(E_{j+1})$.
\end{proposition} 
\begin{proof}
 Since of all coeficientes of $s(x)$ of index less than $m$ vanish, 
 we have that $P= P_{m-1}$ and $\Top (E_m) = \Top(E_{m-1, m/n})=\Top (E_{P, m/n})$.
 Recall that, for any element $E_{k}$, we have $\Bot(E_{k})\geq 1$ by Remark \ref{rem:top-less-than-bottom}.
 If there are no dicritical characteristic exponents, then the result
 follows from Proposition \ref{pro:descent-step}, as
 \begin{equation*}
 1\leq \Bot(E_{{e_g}}) \leq
 \frac{\Top(E_{m}) }{r_1\cdots r_g} .
 \end{equation*}

 Otherwise, let $e_{\ell}$ be the dicritical characteristic exponent with greatest index. By Proposition \ref{pro:descent-step}, we know that
 \begin{equation*}
 \Bot(E_{e_{\ell}}) \leq \frac{\Top(E_{e_{\ell}})}{r_{\ell}} + \frac{r_{\ell}-1}{r_{\ell}}
 \end{equation*}
 and an iterative use of the same Proposition, Lemma
 \ref{lem:inter-dicritical}
 and Corollary~\ref{cor:two-consecutive-dicritical}, if needed, gives
 \begin{equation*}
 \Top(E_{e_{\ell}}) \leq \frac{\Top(E_{m})}{r_1\cdots r_{\ell-1}}.
 \end{equation*}
 Combining both inequalities, we obtain
 \begin{equation*}
 \Bot(E_{e_{\ell}}) \leq \frac{\Top(E_{m})}{r_1\cdots r_{\ell}} + \frac{r_{\ell}-1}{r_{\ell}} . 
 \end{equation*}
 Finally, using Proposition \ref{pro:descent-step} for $\Bot(E_{e_{\ell}})$, taking into account that there are no more dicritical exponents,
 \begin{equation*}
 \Bot(E_{e_{\ell}}) \geq \Bot(E_{e_g})r_{\ell+1}\cdots r_g \geq r_{\ell+1}\cdots r_g.
 \end{equation*}
 Thus, we obtain\ 
 \begin{equation}
 \Top(E_{m}) \geq \Bot(E_{e_g}) r_1 \hdots r_g
 - \frac{r_{\ell} -1}{r_{\ell}} r_1 \hdots r_\ell
 \end{equation}
 and (\ref{equ:ineq}) follows from $\Bot(E_{e_g}) \geq 1$ and $r_g\geq 2$.

 Assume that the strict inequality does not hold, that is
 $\Top(E_{m})= r_1\cdots r_{g-1}$. This implies that $\ell = g$ and
 $\Bot(E_{e_g}) = 1$. There cannot be more dicritical exponents other than $e_{g}/n$ in $\{m,\ldots, e_g \}/n$, since
 Lemma \ref{lem:inter-dicritical} provides strict inequalities. Moreover, we get
 $\Top(E_{e_g}) = r_g \Bot(E_{e_g}) - (r_g-1)=r_g - (r_g-1)=1$. Finally,
 an iterative use of Proposition \ref{pro:descent-step} gives
 $\Top(E_{j})/ \rho_j =\Bot(E_{j})$ and hence $\Bot(E_j)=\Top(E_{j+1})$ for any
 $m \leq j < e_g$.
 \end{proof}
Taking into account that $r_i\geq 2$ for $i=1,\ldots g$, we get:
\begin{corollary}\label{cor:logarithm}
 With the same notation as in Theorem \ref{the:main}, 
 \begin{equation*}
 g \leq 1 + \log_2(H(P,s(x))) \leq 1 + \log_2(H(P)) ,
 \end{equation*}
and if $\ord (s(x)) \geq 1$, then $g\leq 1+\log_2(\nu_0(P)+1)$.
\end{corollary}

\section{Multiplicity and height}
In the case of singular holomorphic foliations (see \cite{ilyashenko-yakovenko} for
instance), given a Pfaffian $1$-form $\omega=A(x,y)dx+B(x,y)dy$ with
$A(x,y),B(x,y)\in \mathbb{C}\{x,y\}$ satisfying $A(0,0)=B(0,0)=0$ and $\gcd(A(x,y),B(x,y))=1$ that defines a germ of
holomorphic foliation $\mathcal{F}$ at $0 \in \mathbb{C}^2$, the 
\emph{multiplicity} of $\mathcal{F}$ at $0$ is
\[ \nu_{0}(\mathcal{F}) = \min \{ \ord_{(x,y)} (A(x,y)),\ord_{(x,y)} (B(x,y)) \}, \] whose value does not depend on the coordinates $(x,y)$
and is at least~$1$. In \cite{Cano-Fortuny-Ribon-2020}, we
proved, using geometric arguments, that, if $\Gamma$ is an invariant irreducible curve,
which in generic coordinates has characteristic exponents $e_1, \ldots, e_g$ and Puiseux factors
$r_1,\ldots, r_g$, then the inequality
\begin{equation*}
 \nu_{0}(\mathcal{F}) \geq r_1\cdots r_{g-1}
\end{equation*}
holds. The natural question is wether we can
obtain this result using just the Newton polygon. The answer is
\emph{yes}. Set $P=A(x,y)+B(x,y) y_1$ and let $s(x)$ be the Puiseux series
expansion corresponding to $\Gamma$. Since $\Gamma$ is an integral
curve of $\omega$, then $s(x)$ is a solution of the
differential equation $P=0$.
After a linear change of coordinates, we can assume that $\Gamma$ is not tangent
to $x=0$; in particular we obtain $\ord_x(s(x))\geq
1$.
Furthermore, after a generic linear change of coordinates we can also assume the
following facts, depending on whether or not $E_{P,1}$ is dicritical,
which corresponds to $\mathcal{F}$ having an invariant branch tangent
to any generic
line in~$(\mathbb{C}^{ 2},0)$: 
\begin{enumerate}
\item Either $\Phi_{P,1} (C)\equiv 0$, or in other words, every line through the origin is an invariant curve of $A_{\nu}(x,y)dx+B_{\nu}(x,y)dy=0$,
where 
$A_{\nu}$ and $B_{\nu}$ are the homogeneous components of degree $\nu = \nu_{0}(\mathcal{F})$
of $A$ and $B$ respectively. This is the dicritical case.
\item Or $x=0$ is not an invariant curve of $A_{\nu}(x,y)dx+B_{\nu}(x,y)dy=0$,
or equivalently $\deg (\Phi_{P,1}(C)) = \nu +1$, the non-dicritical case.
\end{enumerate} 
Those assumptions do not change $\nu_0(\mathcal{F})$. 
We have the inequality
\begin{equation*}
 \nu_0(P)=\nu_0(\mathcal{F}) \geq \Top (E_{P,1}) -1
\end{equation*}
by definition, and by Proposition \ref{pro:b},
\begin{equation*} 
 \nu_0(\mathcal{F}) \geq \Top (E_{P,1}) -1 \geq r_1\cdots r_{g-1} -1.
\end{equation*}
We now prove that the trailing $-1$ can be removed from the inequality, just using arguments from the Newton construction.
\begin{proof}[Proof of Corollary \ref{cor:multiplicity-inequality}]
 The foliation $\mathcal{F}$ is singular, so that $\nu_0(\mathcal{F})\geq 1$ and the result holds for $g\leq 1$. We assume henceforward that $g>1$.
 
Assume, aiming at contradiction, that $ \Top (E_{P,1}) = r_1\cdots r_{g-1}$.
By Proposition~\ref{pro:b}
 we infer that $e_{g}/n$ is the unique dicritical exponent in $\mathbb{Z}_{\geq 1}/n$
 and the element
 $E_{e_{g}-1,e_g/n}$ is dicritical. Since the exponent $1$ is not a
 characteristic exponent the element $E_{P,1}=E_{n}$ of
 co-slope $1$ is non-dicritical and, by assumption (2), there is a point with abscissa
 $-1$ in $E_{n}$, which is indeed $(-1, \nu_0(\mathcal{F})+1)$. As its abscissa is
 $-1$, it is the topmost vertex of $E_{n}$, and we have $\tres_{n} =0$. Moreover, we have
 $\Bot(E_{j})= \Top(E_{{j+1}})$ and hence
 $\bres_{j} = \tres_{j+1}$ for any $n \leq j < e_g$ by Proposition~\ref{pro:b}.
 Furthermore, we get $\Bot(E_{e_{g}})=\Top(E_{e_g})=1$,
 $\Bot(E_{j})=\Top(E_{j})/\rho_j$ and $\Bot(E_{j})=\Top(E_{j+1})$ for any
 $n \leq j < e_g$.
 Applying again Proposition \ref{pro:b} and
 Corollary \ref{cor:residue-in-pure-chain}, we obtain
 \[
 \tres_{j+1} = \bres_{j} =\rho_{j} \tres_{j} +
 (\rho_{j}-1) \delta_{j/n}
 \] 
 for any $n \leq j < e_g$. Since $\tres_{n} =0$, we obtain
 $\tres_{j} \geq 0$ for any $n \leq j \leq e_g$ by induction.
 As $\tres_{e_g} \geq 0$, the top vertex of $E_{e_g}$
 cannot belong to a dicritical element, contradicting that
 $E_{e_g -1, e_{g}/n}$ is dicritical.
\end{proof}
\color{black}

\end{document}